\documentclass{gtmon_a}
\pdfoutput=1

\usepackage{graphicx}

\proceedingstitle{The Zieschang Gedenkschrift}
\conferencestart{5 September 2007}
\conferenceend{8 September 2007}
\conferencename{Conference in honour of Heiner Zieschang}
\conferencelocation{Toulouse, France}

\editor{Michel Boileau}
\givenname{Michel}
\surname{Boileau}

\editor{Martin Scharlemann}
\givenname{Martin}
\surname{Scharlemann}

\editor{Richard Weidmann}
\givenname{Richard}
\surname{Weidmann}

\title{Tangential LS--category of $K(\pi,1)$--foliations}

\author{Wilhelm Singhof}
\givenname{Wilhelm}
\surname{Singhof}
\address{Mathematisches Institut\\ Heinrich-Heine-Universit\"at
D\"usseldorf\\\newline
Universit\"atsstr 1\\ 40225 D\"usseldorf\\
Germany\vspace{3pt}\\\newline
Mathematisches Institut\\
Freie Universit\"at Berlin\\\newline
Arnimallee 3\\14195 Berlin\\ Germany}
\email{vogt@math.fu-berlin.de}
\email{singhof@math.uni-duesseldorf.de}
\urladdr{}

\author{Elmar Vogt}
\givenname{Elmar}
\surname{Vogt}
\urladdr{}

\volumenumber{14}
\issuenumber{}
\publicationyear{2008}
\papernumber{021}
\startpage{477}
\endpage{504}

\doi{}
\MR{}
\Zbl{}

\arxivreference{} 



\received{31 May 2006}
\revised{26 July 2006}
\accepted{31 May 2007}
\published{29 April 2008}
\publishedonline{29 April 2008}
\proposed{}
\seconded{}
\corresponding{}
\version{}

\makeatletter
\def\cnewtheorem#1[#2]#3{\newtheorem{#1}{#3}[section]
\expandafter\let\csname c@#1\endcsname\c@thm}
\makeatother

\AtBeginDocument{\let\bar\wbar\let\tilde\wtilde\let\hat\what}
\newcommand{\circU}{\smash{\stackrel{\circ}{\smash{U}\vrule height6pt width0pt depth0pt}}}
\newcommand{\dDelta}{\smash{\stackrel{.}{\smash{\Delta} \vrule height6.5pt width0pt depth0pt}}}
\newcommand{\ddDelta}{\smash{\stackrel{..}{\smash{\Delta} \vrule height6.5pt width0pt depth0pt}}}
\newcommand{\circD}{\mskip3mu\smash{\stackrel{\circ}{\mskip-3mu\smash{D}\vrule height5pt width0pt depth0pt}}^2}
\newcommand{\slra}[1]{\,\smash{\stackrel{#1}{\smash{\lra}\vrule height3.5pt depth0pt width0pt}}\,}

\newtheorem{thm}{Theorem}[section] 
\cnewtheorem{lem}[thm]{Lemma} 
\cnewtheorem{cor}[thm]{Corollary} 
\cnewtheorem{prop}[thm]{Proposition} 

\theoremstyle{definition}
\cnewtheorem{defn}[thm]{Definition} 
\newtheorem*{rem}{Remark} 
\newtheorem*{notation}{Notation} 
\cnewtheorem{example}[thm]{Example} 
\cnewtheorem{problem}[thm]{Problem} 
\makeautorefname{defn}{Definition}

\newcommand{\smetmi}{\smallsetminus}
\makeop{cat}
\makeop{codim}
\newcommand{\geodim}{{\rm geom.dim\,}}
\makeop{id}

\newcommand{\Sum}{\sum\limits}

\newcommand{\cB}{\mathcal{B}}
\newcommand{\cC}{\mathcal{C}}

\newcommand{\cF}{\mathcal{F}}
\newcommand{\cL}{\mathcal{L}}
\newcommand{\cP}{\mathcal{P}}
\newcommand{\cQ}{\mathcal{Q}}
\newcommand{\cR}{\mathcal{R}}
\newcommand{\cU}{\mathcal{U}}

\newcommand{\dR}{\frak{R}}
\newcommand{\dU}{\frak{U}}
\newcommand{\N}{{\mathbb{N}}} 
\newcommand{\ra}{\rightarrow}
\newcommand{\lra}{\longrightarrow}

\begin{document}

\begin{asciiabstract} 
A K(pi,1)-foliation is one for which the universal covers of all leaves 
are contractible (thus all leaves are K(pi,1)'s for some pi). In the first part
of the paper we show that the tangential Lusternik--Schnirelmann
category cat F of a K(pi,1)-foliation F on a manifold M is bounded from
below by
t-codim F for any t with H_t(M;A) nonzero for some coefficient group A.
Since for any
C^2-foliation F one has cat F <= dim F by our Theorem 5.2 of [Topology 42 (2003) 603-627], this
implies that
cat F = dim F for  K(pi,1)-foliations of class C^2 on closed manifolds.

For K(pi,1)-foliations on open
manifolds the above estimate is far from optimal, so one might hope for some other
homological lower bound for cat F. In the second part we see that foliated
cohomology will not work. For we show that the p-th foliated cohomology group of a
p-dimensional foliation of positive codimension  is an infinite dimensional vector
space, if the foliation is obtained from a foliation of a manifold by removing
an appropriate closed set, for example a point. But there are simple examples of
K(pi,1)-foliations of this type with cat F < dim F. Other, more interesting
examples of  K(pi,1)-foliations on open manifolds are provided by the finitely 
punctured Reeb foliations on lens spaces whose tangential category we calculate.

In the final section we show that  C^1-foliations of tangential
category at most 1 on closed  manifolds are locally trivial homotopy
sphere bundles. Thus among 2-dimensional C^2-foliations  on closed
manifolds the only ones whose tangential category is still unknown are those which are
2-sphere bundles which do not admit sections.
\end{asciiabstract}

\begin{htmlabstract}
<p class="noindent">
A K(&pi;,1)&ndash;foliation is one for which the universal covers of all
leaves are contractible (thus all leaves are K(&pi;,1)'s for some
&pi;). In the first part of the paper we show that the tangential
Lusternik&ndash;Schnirelmann category cat F of a
K(&pi;,1)&ndash;foliation F on a manifold M is bounded from
below by t - codim F for any t with H<sub>t</sub>(M;A)&ne; 0 for
some coefficient group A.  Since for any C<sup>2</sup>&ndash;foliation
F one has cat F&le; dim F by our
earlier work, [Topology 42 (2003) 603-627; Theorem 5.2], this implies
that cat F = dim F for K(&pi;,1)&ndash;foliations of
class C<sup>2</sup> on closed manifolds.
</p>
<p class="noindent">
For K(&pi;,1)&ndash;foliations on open manifolds the above estimate is far
from optimal, so one might hope for some other homological lower bound
for cat F.  In the second part we see that foliated
cohomology will not work.  For we show that the p-th foliated
cohomology group of a p&ndash;dimensional foliation of positive
codimension is an infinite dimensional vector space, if the foliation
is obtained from a foliation of a manifold by removing an appropriate
closed set, for example a point.  But there are simple examples of
K(&pi;,1)&ndash;foliations of this type with cat F<
dim F.  Other, more interesting examples of
K(&pi;,1)&ndash;foliations on open manifolds are provided by the finitely
punctured Reeb foliations on lens spaces whose tangential category we
calculate.
</p>
<p class="noindent">
In the final section we show that C<sup>1</sup>&ndash;foliations of tangential
category at most 1 on closed manifolds are locally trivial homotopy
sphere bundles. Thus among 2&ndash;dimensional C<sup>2</sup>&ndash;foliations on
closed manifolds the only ones whose tangential category is still
unknown are those which are 2&ndash;sphere bundles which do not admit
sections.
</p>
\end{htmlabstract}

\begin{abstract} 
A $K(\pi,1)$--foliation is one for which the universal covers of all leaves 
are contractible (thus all leaves are $K(\pi,1)$'s for some $\pi$). In the first part
of the paper we show that the tangential Lusternik--Schnirelmann
category $\cat\cF$ of a $K(\pi,1)$--foliation $\cF$ on a manifold $M$ is bounded from
below by
$t-\codim \cF$ for any $t$ with $H_t(M;A)\neq 0$ for some coefficient group $A$.
Since for any
$C^2$--foliation $\cF$ one has $\cat\cF\leq$ $\dim\cF$ by our earlier work \cite[Theorem~5.2]{SinV}, this
implies that
$\cat\cF = \dim\cF$ for  $K(\pi,1)$--foliations of class $C^2$ on closed manifolds.

For $K(\pi,1)$--foliations on open
manifolds the above estimate is far from optimal, so one might hope for some other
homological lower bound for $\cat\cF$. In the second part we see that foliated
cohomology will not work. For we show that the $p$--th foliated cohomology group of a
$p$--dimensional foliation of positive codimension  is an infinite dimensional vector
space, if the foliation is obtained from a foliation of a manifold by removing
an appropriate closed set, for example a point. But there are simple examples of
$K(\pi,1)$--foliations of this type with $\cat\cF<$ $\dim\cF$. Other, more interesting
examples of  $K(\pi,1)$--foliations on open manifolds are provided by the finitely 
punctured Reeb foliations on lens spaces whose tangential category we calculate.

In the final section we show that  $C^1$--foliations of tangential
category at most 1   on closed  manifolds are locally trivial homotopy
sphere bundles. Thus among $2$--dimensional $C^2$--foliations  on closed
manifolds the only ones whose tangential category is still unknown are those which are
$2$--sphere bundles which do not admit sections.
\end{abstract}

\begin{webabstract} 
A $K(\pi,1)$--foliation is one for which the universal covers of all
leaves are contractible (thus all leaves are $K(\pi,1)$'s for some
$\pi$). In the first part of the paper we show that the tangential
Lusternik--Schnirelmann category $\mathrm{cat} \mathcal{F}$ of a
$K(\pi,1)$--foliation $\mathcal{F}$ on a manifold $M$ is bounded from
below by $t-codim \mathcal{F}$ for any $t$ with $H_t(M;A)\neq 0$ for
some coefficient group $A$.  Since for any $C^2$--foliation
$\mathcal{F}$ one has $cat \mathcal{F}\leq$ $dim \mathcal{F}$ by our
earlier work, [Topology 42 (2003) 603-627; Theorem 5.2], this implies
that $cat \mathcal{F} = dim \mathcal{F}$ for $K(\pi,1)$--foliations of
class $C^2$ on closed manifolds.

For $K(\pi,1)$--foliations on open manifolds the above estimate is far
from optimal, so one might hope for some other homological lower bound
for $cat \mathcal{F}$.  In the second part we see that foliated
cohomology will not work.  For we show that the $p$-th foliated
cohomology group of a $p$--dimensional foliation of positive
codimension is an infinite dimensional vector space, if the foliation
is obtained from a foliation of a manifold by removing an appropriate
closed set, for example a point.  But there are simple examples of
$K(\pi,1)$--foliations of this type with $cat \mathcal{F}<$
$dim \mathcal{F}$.  Other, more interesting examples of
$K(\pi,1)$--foliations on open manifolds are provided by the finitely
punctured Reeb foliations on lens spaces whose tangential category we
calculate.

In the final section we show that $C^1$--foliations of tangential
category at most 1 on closed manifolds are locally trivial homotopy
sphere bundles. Thus among $2$--dimensional $C^2$--foliations on
closed manifolds the only ones whose tangential category is still
unknown are those which are $2$--sphere bundles which do not admit
sections.
\end{webabstract}

\maketitle


\setcounter{section}{-1}
\section{Introduction}\label{Introduction}
A subset $U$ of a topological space $X$ is
called categorical (in the sense of Lusternik and Schnirelmann) if $U$
is open and the inclusion $U\subset X$ is homotopic to a constant
map.  The \textit{Lusternik--Schnirelmann category} $\cat X$ of $X$ is the
least number $r$ such that $X$ can be covered by $r+1$ categorical
sets.

The Lusternik--Schnirelmann category, LS--category for short, is a homotopy
invariant.  This follows directly from its definition.  In general, one
obtains upper bounds by constructing categorical covers.  Nontrivial
lower bounds are quite often very hard to obtain, and this makes the
computation of $\cat X$ a difficult task.  For example, only quite
recently N\,Iwase developed in a series of papers methods to determine
the LS--category of the total space of sphere bundle over spheres
\cite{I1,I2}.  These are CW--complexes with at most four
cells, and thus their LS--category is 1, 2, or 3.

Much earlier, in a very short paper \cite{EilGan}, Eilenberg and Ganea
state without proof three propositions from which they establish the
LS--categories of $K(\pi,1)$--spaces apart from a few low-dimensional
cases.  To do this they compare $\cat\pi$, the LS--category of such a
space, with two other invariants of $\pi$:  its cohomological
dimension, $\dim\pi$, and its geometric dimension, $\geodim\pi$.  The
last one is the smallest $n$ such that there exists an
$n$--dimensional CW--complex which is a $K(\pi,1)$.  Clearly,
$\dim\pi\leq\geodim\pi$ and $\cat\pi\leq$ $\geodim\pi$.  The
statements in \cite{EilGan} are more general, but if we exclude groups
of cohomological dimension less than 3, Proposition~2 of \cite{EilGan}
states that $\geodim\pi\leq$ $\dim\pi$, and Proposition~3 states that
$\dim\pi\leq$ $\cat\pi$.  Thus, if $\dim\pi\geq3$, then $\cat\pi=$
$\dim\pi=$ $\geodim\pi$.  Also, for an $n$--dimensional aspherical
CW--complex $X$ with $H_n(X;A)\neq0$ for some abelian group $A$ we
have that $\cat X=n$.  This follows from Proposition~3 alone.  We want to
generalize this result to the case of foliations.

For foliations, the concept of LS--category was introduced by Hellen
Colman in her thesis \cite{Colthesis} (see also Colman and Macias-Virg{\'o}s \cite{ColMac,Coltan}).  
Depending on whether the transverse or tangential
aspect of the foliation is of more interest there are the concepts of
(saturated) transverse and tangential LS--categories.  We are concerned
only with the latter.

\begin{defn}\label{def:catset}
A subset $U$ of a manifold $M$ with foliation $\cF$ is called 
\textit{tangentially categorical} if it is open and there exists a homotopy
$h\co U\times I\lra M$ with the following properties:
\begin{itemize}
\item[(1)] $h_0\co  U\ra M$ is the embedding
  $U\hookrightarrow M$. 
\item[(2)] For each $x\in U$ the path $t\mapsto h(x,t)$,
$t\in I$, is contained in a leaf of $\cF$.
\item[(3)] If $\cF_U$ denotes the restriction of
$\cF$ to $U$ then $h_1$ maps each leaf of $\cF_U$ to a point.
\end{itemize}
\end{defn}

\begin{defn}\label{def:cat}
Let $\cF$ be a foliation of a manifold $M$.  The \textit{tangential
LS--category} of $\cF$, $\cat\cF$ for short, is the least integer $r$
such that $M$ can be covered by $r+1$ tangentially categorical sets.
\end{defn}

\begin{notation} A foliation will be called a \textit{$K(\pi,1)$--foliation} if the universal
cover of every leaf is contractible.
\end{notation}

Our main result is the following:
\begin{thm}\label{main}
Let $\cF$ be a $p$--dimensional $K(\pi,1)$--foliation of the $n$--manifold
$M$.  Assume that for
some abelian group $A$ and integer $t$ the group $H_t (M;A)\neq0$.  Then
$\cat\cF\geq t-(n-p) = t-\codim\cF$.
\end{thm}

It is known by our earlier work \cite[Theorem 5.2]{SinV} that $\cat\cF\leq\dim\cF$ for
$C^2$--foliations $\cF$.  So we have:

\begin{cor}\label{cor:main}
Let $\cF$ be a $K(\pi,1)$--foliation of class $C^2$ on
a closed manifold.  Then $\cat\cF=\dim\cF$.
\end{cor}

For foliations $\cF$ on open manifolds the lower bound provided by the
above theorem is far from optimal.  Consider for example the Reeb
foliation $\cR$ of $S^3$ and remove a point $y$ from $S^3$ which does
not lie on the toral leaf of $\cR$.  Call the resulting foliation
$\cR_y$.  Since the ordinary LS--category of each leaf is a lower bound
for the LS--category of the foliation we have $\cat\cR_y=2$ while
$S^3\smetmi\{y\}$ is contractible.

We obtain Proposition~3 of \cite{EilGan} for countable groups $\pi$ of finite
cohomological dimension by applying our theorem to a foliation with a single
leaf.  Also note that by results of Haefliger \cite{HaeAst} our
hypotheses imply 
that the manifold $M$ is a classifying space for the fundamental groupoid
$\Pi_\cF$ of $\cF$ (see \fullref{sec:tgroup}). So $(M,\cF)$ can be
regarded as a foliated $K(\Pi_\cF,1)$.  Then our lower bound for
$\cat\cF$ is homological $\dim\Pi_\cF -\codim\cF$.  Analyzing the most
likely proof of Proposition~3 in \cite{EilGan} and the definition of
$\cat\cF$, another potential lower bound for $K(\pi,1)$--foliations
comes to mind:  the smallest number $s$ such that $H_{\cF}^k (M)=0$
for all $k>s$.  Here $H_{\cF}^k (M)$ is the foliated de Rham  cohomology
of the foliation $\cF$. This could be enhanced by adding some foliated local
coefficient system.

This number would be perfectly suited to deal with the example $\cR_y$
above. But we will show in Sections \ref{section:folcoh} and
\ref{section:Reeb} that in general it is not a lower bound for $\cat\cF$. It
is not hard to show (see \fullref{section:folcoh}) that after removal of
a point
$x$ from a manifold $M$ with a $p$--dimensional foliation $\cF$ we have
$H^p ({\cF_x})\neq0$ for the induced foliation $\cF_x$  on
$M\smallsetminus \{x\}$; in fact it is infinite dimensional. Let $\cF$ be the foliation
of $\R^3$ by horizontal planes, and $\cF_0$ the induced foliation of $\R^3\smallsetminus
\{0\}$. It is easy to see that $\cat\cF=1$.  So the
foliated homological dimension is not a lower bound for $\cat\cF$. Another example is
the punctured Reeb foliation $\cR_y$.  
We will prove in \fullref{section:Reeb} that $\cat\cR_x =1$ if $x$ is a point on the
toral leaf of $\cR$.  

Furthermore, on closed manifolds, where by our main result foliated cohomological
dimension is obviously a lower bound, it sometimes fails to be optimal. This is shown by 
Colman and Hurder, who prove in
\cite{ColHurTan} that
$H_{\cF}^2 (M)=0$ for the stable (and unstable) foliations $\cF$ of Anosov flows on
closed $3$--manifolds $M$. Since the leaves of these foliations are cylinders or planes they
are $K(\pi,1)$--foliations, and thus have category 2. 

    Here is a brief outline of the paper. In \fullref{sec:tgroup} we associate to a tangentially categorical open cover of a foliated manifold $M$ a topological groupoid, called the transverse fundamental groupoid of the foliation and prove that its classifying space is naturally weakly homotopy equivalent to $M$ if the foliation is a $K(\pi,1)$--foliation. In \fullref{sec:specseq} we use the spectral sequence associated to the filtration of classifying spaces related to their construction as the thick realization (see Segal \cite{Segal,Segal:cohomology}) of a simplicial set to prove \fullref{main}. As mentioned above we show in \fullref{section:folcoh} that foliated cohomological dimension is not a lower bound for the tangential category of $K(\pi,1)$--foliations. \fullref{section:Reeb} contains a study of the tangential category of various punctured Reeb foliations, the result depending on where the punctures lie. Finally, in \fullref{sec:cat1} we show that the leaves of any $C^1$--foliation of dimension at least 2 and category at most 1 on a closed manifold are the fibres of a homotopy sphere bundle. 
    
\medskip\textbf{Remark on smoothness hypotheses}\qua In a few claims we make the hypothesis that the foliations are of class $C^2$ or $C^1$. This is due to the fact that we use results from other papers where these results are proved under these assumptions, or, as in \fullref{prop:accum}, to be able to make use of the simple techniques available for differentiable manifolds. Whether these assumptions are really necessary in each instance, we have not checked.

The idea of using the (co)homology of certain classifying spaces for
obtaining lower bounds for $\cat\cF$ is due to Colman and Hurder
based on earlier work of Shulman  on covering 
dimensions of foliation atlases \cite{Shul}. They obtain in
\cite{ColHurTan} (among many other things)  lower bounds by exploiting
the nonvanishing of secondary characteristic classes of $\cF$
\cite[Theorem 5.3]{ColHurTan}. \fullref{main} above generalizes Theorem 7.5 of
\cite{ColHurTan}.

\section{The transverse fundamental groupoid associated to a
tangentially categorical cover}
\label{sec:tgroup}

In this section $\cF$ will be a $p$--dimensional $C^0$--foliation of an
$n$--manifold $M$.  Let $(U_j)_{j\in J}$ be a locally finite
tangentially categorical cover of $M$, and for $j\in J$ let $h_j\co  U_j
\times I \ra M$ be a homotopy satisfying \fullref{def:catset} (1)--(3).  As
usual $h_{jt}\co U_j \ra M$ is the map defined by $h_{jt} (x) = h_j (x,t)$.

To $(U_j,h_j)_{j\in J}$ we will associate a topological groupoid which
will be called the transverse fundamental groupoid of $\cF$ associated
to $(U_j,h_j)_{j\in J}$.

For $j\in J$ let $T_j$ be the space of leaves of the restriction
$\cF_j$ of $\cF$ to $U_j$.  By \mbox{\cite[Lemma~1.1]{SinV}},  each $T_j$ is an
$(n-p)$--dimensional manifold which may be non-Hausdorff.  Let $T:=
\bigsqcup_{j\in J} T_j$.  For each leaf $f\in T_j$, let $c(f)$
be the image of $f$ by $h_{j1}\co  U_j \ra M$.  By \fullref{def:catset} (3) this is a point
in the leaf of $\cF$ containing $f$.  The map
$$
c\co  T\ra M
$$
is a continuous immersion.

The transverse fundamental groupoid of $\cF$ associated to
$(U_j,h_j)_{j\in J}$ will be denoted by $\Pi_T$ for short.  The
elements of $\Pi_T$ are all triples $(f,[\gamma],g)$ with $f,g\in T$
and $[\gamma]$ a leafwise path homotopy class of a path $\gamma$ in a
leaf of $\cF$ beginning in $c(g)$ and ending in $c(f)$.  Composition is
the obvious one:
$$
(f,[\gamma],g)\cdot (g,[\gamma'],h) := (f,[\gamma'\ast\gamma],h),
$$
where $\ast$ denotes the usual path multiplication.  The units of
$\Pi_T$ are the elements of $T$.

There is a natural topology on $\Pi_T$ which makes $\Pi_T$ an
$(n-p)$--dimensional manifold which may be non-Hausdorff.  For
$(f,[\gamma],g) \in \Pi_T$ basic neighborhoods can be described by
lifting a representative of $[\gamma]$ in a continuous way into
neighboring leaves as is done in foliation theory to define holonomy
along a path, and moving $f,g$ accordingly.  In more detail, choose a
representative $\gamma$ of $[\gamma]$ and points $y'\in f$, $y\in
g$.  Let $f$ be in $T_i$ and $g$ in $T_j$.  For $k\in J$, $z\in U_k$,
denote by $\nu_{z,k}$ the path $t\mapsto h_k (z,t)$, and by
$\cF_{k,z}$ the leaf of $\cF_k$ through $z$.  Consider the leaf path
$\eta_y = \nu_{y,j} \ast \gamma\ast\nu_{y',i}^{-1}$, a compact
neighborhood $K$ of the image of $\eta_y$ in the leaf of $\cF$
containing $y$, and a tubular neighborhood $\psi\co  N_K \ra K$ of $K$
with fibres transverse to $\cF$.  By \cite[Corollary~6.21]{Sie}, such a
tubular neighborhood exists also for $C^0$--foliations.  If $D$ is a
sufficiently small neighborhood of $y$ in $\psi^{-1}(y)$, then for every
$z\in D$ there exists exactly one leaf path $\eta_z$ contained in $N_K$
such that
$\eta_y =
\psi\circ\eta_z$ and $\eta_z(0)=z$.  Let $\eta_z(1)=z'$.  We may assume
that $D\subset U_j$, and if $D$ is small enough, also that $D' = \{z':  z\in
D\} \subset U_i$, and that $D'$, $D$ project homeomorphically onto
open sets of $T_i$ respectively $T_j$.  (See again Lemma~1.1 of
\cite{SinV}).

The sets of the form
\begin{equation}\label{not:neighborhood}
N(E) := \big\{(\cF_{iz'}, [\nu_{z,j}^{-1} \ast \eta_z \ast
\nu_{z',i}], \cF_{jz}):  z\in E\big\}\,,
\end{equation}
with $E$ an open neighborhood of $y$ in $D$ form a neighborhood basis
of $(f,[\gamma],g)$ defining the desired topology on $\Pi_T$. With
this topology $N(E)$ is homeomorphic to $E$. Thus $\Pi_T$ is a (not
necessarily Hausdorff) $(n-p)$--manifold. A bit more generally, we have:

\begin{prop}\label{prop:composables}
Let $\Pi_{T,m}$ be the space of $m$--fold composable elements of
$\Pi_T$, ie, $\Pi_{T,m} = \{(a_1,\ldots,a_m)\in(\Pi_T)^m : a_1
\cdot\ldots\cdot a_m \textrm{ exists }\}$. Then $\Pi_{T,m}$ is an
$(n-p)$--manifold which may be non-Hausdorff.
\end{prop}

\begin{proof}
This is similar to the proof that the space of $m$--fold composables in
the groupoid of germs of local diffeomorphisms of an $(n-p)$--manifold
is itself a (not necessarily Hausdorff) $(n-p)$--manifold. (see, for example,
Bott \cite{Bott}). We leave the (easy) details to the reader.
\end{proof}

Note that $\Pi_T$ and $\Pi_{T,m}$ might be non-Hausdorff even if $T$ is
Hausdorff. 

The tangentially categorical cover $(U_j,h_j)_{j\in J}$ gives rise to a
topological groupoid homomorphism $\psi\co  \Gamma_\dU \ra \Pi_T$ where
$\Gamma_\dU$ is the topological groupoid associated to the cover
$(U_j)_{j\in J}$.  As a space $\Gamma_\dU = \bigsqcup_{i,j\in J}
U_i \cap U_j$, so elements correspond to triples $(i,x,j): (j,x,j)\lra (i,x,i)$ with
$x\in U_i \cap U_j$.  The composition $(i,x,j)\cdot (i',x',j')$ is
defined, iff $i'=j$ and $x=x'$, and in this case it is equal to
$(i,x,j')$.  The space of units of $\Gamma_\dU$, ie, identity elements of $\Gamma_\dU$, 
is homeomorphic to
$\bigsqcup_{j\in J} U_j$.  The homomorphism $\psi$ is defined by
$$
\psi (i,x,j) = (\cF_{i,x}, [\nu_{j,x}^{-1} \ast \nu_{i,x}],
\cF_{j,x})\,.
$$
A continuous homomorphism $\Gamma_\dU \ra \Gamma$ from $\Gamma_\dU$
into any topological groupoid $\Gamma$ defines a (representative of a)
$\Gamma$--structure on $M$ in the sense of \cite{Hae58}, or equivalently
a principal $\Gamma$--bundle over $M$ \cite[2.2.2]{HaeAst}.  In the
case of the homomorphism $\psi$ above the $\Pi_T$--bundle is obtained
from the disjoint union $\bigsqcup_{j\in J} U_j \times_\tau \Pi_T$ by
identifying $(x,(f,[\gamma],g))$ in $U_i \times_\tau \Pi_T$ with 
$(y,(f',[\gamma'],g'))$ in $U_j \times_\tau \Pi_T$ if and only if $x=y$,
and $(f',[\gamma'],g') = \psi(j,x,i) \cdot (f,[\gamma],g)$, ie,  if
$x=y$, $g=g'$, and $[\gamma'] = [\gamma\ast\nu_{i,x}^{-1} \ast
\nu_{j,x}]$.  Here, $(x,(f,[\gamma],g)) \in U_i \times_\tau \Pi_T$ if
and only if  $(f,[\gamma],g) \in \Pi_T$, $f\in T_i$, and
$x\in f$.  

We denote the corresponding $\Pi_T$--bundle by $E_\psi
\slra{p} M$.  The map $E_\psi \slra{q} T =
\bigsqcup_{j\in J} T_j$ to the units of $\Pi_T$ which is
needed to describe the right action of $\Pi_T$ is given by
$\big[(x,(f,[\gamma],g)\big] \mapsto g$.

Recall \cite[3.2.2]{HaeAst} that a continuous map $f\co  X\ra Y$ between
topological spaces is called a submersion, if for every $x\in X$ there
exists a neighborhood $U$ of $y=f(x)$ in $Y$, a neighborhood $V$ of
$x$ in the fibre $f^{-1}(y)$ of $y$ and a homeomorphism $h\co   V\times U
\ra W$ onto a neighborhood $W$ of $x$ such that $f\circ h$ is the
projection onto the second factor.

The main reason for defining the transverse groupoid $\Pi_T$
associated to $(U_j,h_j)_{j\in J}$ the way we did, especially its
space $T$ of units, is the following proposition.

\begin{prop}\label{prop:submersion}
The map $q\co E_\psi \ra T = \bigsqcup_{j\in J} T_j$ from the
total space of the principal $\Pi_T$--bundle over $M$ associated to
$\psi$ to the space of units of $\Pi_T$ is a submersion. Furthermore,
for any $g\in T$ the fibre $q^{-1}(g)$ of $E_\psi$ over $g$ is the
universal cover of the leaf $L_g$ of $\cF$ which contains $g$.
\end{prop}

\begin{proof}
We first proof the second statement. For a point $y_0 \in L_g$ we
identify the universal cover $(\tilde{L}_g,y_0)$ of $L_g$ associated
to $y_0$ with the space of path homotopy classes of paths in $L_g$
which start in $y_0$. If $g\in T_j$ we choose $y_0 = c(g) :=
h_{j1}(g)$. Maps $a\co q^{-1}(g) \ra (\tilde{L}_g,y_0)$ and $b\co
(\tilde{L}_g,y_0) \ra q^{-1}(g)$ are defined by
\begin{align*}
a\big[(x,(f,[\gamma],g)\big] &= [\gamma\ast\nu_{i,x}^{-1}]
\\
b\big([\rho]\big) &= \big[(\rho(1),(\cF_{k,\rho(1)},
[\rho\ast\nu_{k,\rho(1)}], g))\big]\,.
\end{align*}
Here $(x,(f,[\gamma],g))$ is a representative of an element of
$q^{-1}(g)$ in $U_i \times_\tau \Pi_T$, while $[\rho]$ is an element
of $(\tilde{L}_g,y_0)$ with $\rho(1) \in U_k$.

Because of the identifications made on $\bigsqcup_{i} U_i
\times_\tau \Pi_T$ to obtain $E_\psi$ the maps $a$ and $b$ are well
defined, and are continuous inverses of each other.

To prove the first statement it suffices to consider the restriction
of $q$ to the image of $U_i \times_\tau \Pi_T$ in $E_\psi$, which we
identify with $U_i \times_\tau \Pi_T$. Fix a point
$(y',(f,[\gamma],g)) \in U_i \times_\tau \Pi_T$ and choose a
neighborhood of $(f,[\gamma],g)$ in $\Pi_T$ of the form $N(E)$
in \eqref{not:neighborhood}. We choose $E$ as a small
disk around $y=D\cap g$ in $D$ so that $E' = \{z': z\in E\}$ is a
factor of a foliation neighborhood $V\times E'$ around $y'$, with
$V\times\{y'\}$ a neighborhood of $y'$ in $f$. Then
$$
\big\{ ((v,z'), (\cF_{iz'}, [\nu_{z,j}^{-1} \ast \eta_z \ast
\nu_{z',i}], \cF_{jz})) : v\in V,\, z\in E\big\}
$$
is a neighborhood of $(y',(f,[\gamma],g))$ in $U_i \times_\tau \Pi_T$
diffeomorphic to $V\times E$ and $q$ corresponds to the projection
onto $E$ followed by the embedding $E\hookrightarrow U_j \ra
U_j/\cF_j = T_j \subset T$.

Finally,
$$
\big\{((v,y'), (f,[\gamma],g)): v\in V\big\}
$$
is a neighborhood of $(y',(f,[\gamma],g))$ in $q^{-1}(g)$.
\end{proof}

By \cite[3.2.3]{HaeAst}, the principal $\Pi_T$--bundle $\smash{E_\psi}
\slra{p} M$ is $k$--universal if for every $g\in T$ the space $q^{-1}(g)
\subset E_\psi$ is $(k-1)$--connected.  In
particular, if $\cF$ is a $K(\pi,1)$--foliation, then $E_\psi
\slra{p} M$ is a universal principal $\Pi_T$--bundle. In this
case any map from $M$ into a classifying space for numerable principal
$\Pi_T$--bundles inducing the bundle $E_\psi \slra{p} M$ is a
weak homotopy equivalence.

There exist several constructions for a classifying space of
$\Gamma$--bundles for topological groupoids $\Gamma$, most prominently
the Milnor construction as exposed by Haefliger in
\cite[Section~5]{Hae:integrability}. For our purposes the so-called
thick realization $\|\Gamma\|$ of the associated simplicial space
\cite{Segal:cohomology} will be more convenient.

Recall that the thick realization $\|X\|$ of a simplicial space
$X_\bullet$ is a quotient of 
$$
\bigsqcup\limits_{n\geq0} X_n \times \Delta^n
$$
where $\Delta^n$ is the standard $n$--simplex. The identifications are
given by 
$$
\begin{array}{l}
(x_n, (t_0,\ldots,t_{i-1},0,t_i,\ldots,t_{n-1}))\sim (d_i x_n,
(t_0,\ldots,t_{n-1}))\,,\\[2mm] 
x_n \in X_n\,,\; (t_0,\ldots,t_{n-1}) \in
\Delta^{n-1} \,,\; 0\leq i\leq n\,,
\end{array}
$$
where $d_i \co X_n \ra X_{n-1}$ is the $i$--th face map.

There is a variant of the Milnor construction which does not make use
of inverses (in the groupoid $\Gamma$) and thus can be applied to any
topological category $C$ (see, for example, Stasheff \cite{Stasheff}). We will call it
$\cB C$. As with $\|~~\|$ one first passes to the associated simplicial
space $C_\bullet$. Then one obtains $\cB C$ from 
$$
\bigsqcup\limits_{n\geq0}\; \bigsqcup\limits_{\sigma\in\Sigma_n} C_n
\times \Delta_\sigma^n
$$
by making the obvious face identifications. Here $\Sigma_n$ is the set
of $n$--faces of the standard infinite dimensional simplex
$\Delta^\infty$ in $\R^\infty$.

There is an obvious map $\cB C\ra \|C\|$ which by \cite{tomDieck} is a
homotopy equivalence.

A classifying map $\bar{\psi}\co M {\lra} \cB\Pi_T$ for $p\co E_\psi 
{\lra} M$ is obtained by choosing a partition of unity
$(t_j)_{j\in J}$ for our tangentially categorical cover $(U_j)_{j\in
  J}$ and an ordering of $J$ which we may assume to be a subset of
$\N$. For $x\in M$ let $\{j_0 < j_1 <\cdots< j_k \} = \{j\in J: t_j (x)
>0\}$ and let $\sigma_x$ be the face of $\Delta^\infty$ spanned by
$e_{j_0},\ldots,e_{j_k}$. Then $\bar{\psi}(x)$ is the equivalence class of 
$$
((\psi(j_0,x,j_1),\ldots,\psi(j_{k-1},x,j_k)), \Sum_{i=0}^k t_{j_i}(x)
e_{j_i}) \in \Pi_{T,k} \times \Delta_{\sigma_x}^k \quad \textrm{ in }
\cB\Pi_T\,. 
$$
Of course, for $k=0$, the value of $\bar{\psi}(x)$ is $(\psi(j_0,x,j_0),e_{j_0}).$

Obviously $\bar{\psi}$ factors through the map $\cB(\Gamma_\dU)
\slra{\cB\psi} \cB(\Pi_T)$,  and the corresponding map $M
\slra{u} \cB(\Gamma_\dU)$ defines the  principal
$\Gamma_\dU$--bundle $\bigsqcup_iU_i \ra M$. This is universal since the associated map to
the space of units of $\Gamma_{\mathfrak{U}}$ is the identity.

As a result of this discussion we know the following. If $\cF$
is a $K(\pi,1)$--foliation, then all maps in the following diagram are weak homotopy
equivalences, and the vertical maps are even  homotopy equivalences.
$$
\setlength{\unitlength}{1mm}
\begin{picture}(50,30)(0,-5)
\put(0,10){\makebox(0,0){$M$}}
\put(3,8){\vector(2,-1){12}}
\put(3,12){\vector(2,1){12}}
\put(5,3){\makebox(0,0){$\pi\circ u$}}
\put(7,16){\makebox(0,0){$u$}}
\put(21,2){\makebox(0,0){$\|\Gamma_\dU\|$}}
\put(21,18){\makebox(0,0){$\cB\Gamma_\dU$}}
\put(20,16){\vector(0,-1){11}}
\put(18,11){\makebox(0,0){$\pi$}}
\put(28,2){\vector(1,0){15}}
\put(28,18){\vector(1,0){15}}
\put(35,5){\makebox(0,0){$\|\psi\|$}}
\put(35,20){\makebox(0,0){$\cB\psi$}}
\put(50,2){\makebox(0,0){$\|\Pi_T\|$}}
\put(50,18){\makebox(0,0){$\cB\Pi_T$}}
\put(50,16){\vector(0,-1){11}}
\end{picture}
$$
For later use we want to replace $\|\Gamma_\dU\|$ by a smaller
model. Let $\cU$ be the subcategory of $\Gamma_\dU$ having the same
objects but only admitting a morphism from $(i,x,i)$ to $(j,y,j)$ if
and only if $x=y$ and $i\leq j$ in the ordering of $J$; there is then
as before exactly one morphism. We call $\cU$ the category associated
to the ordered cover $(U_j)_{j\in J}$. By construction $\pi\circ u$
factors as $M\slra{\bar{u}} \|\cU\| \lra \|\Gamma_\dU\|$,
where the second map is induced by the inclusion $\cU\ra\Gamma_\dU$.
It is well known that $\bar{u}$ is a homotopy equivalence. One sees this by passing to
the thin realization $|\cU |$ of $\cU$. The thin realization of a simplicial space
$X_\bullet$ is obtained from $\bigsqcup_{n\geq0} X_n \times \Delta^n$ by
considering both, face and degeneracy operators, when making the identifications. This is
the realization introduced by G\,Segal \cite{Segal}. In tom Dieck \cite{tomDieck} it is called
the geometric realization. The canonical projection $\| X_\bullet\| \lra |X_\bullet |$ is
a homotopy equivalence, if the inclusion of the degenerate simplices into the space of
all simplices is a cofibration \cite[Proposition~1]{tomDieck}, \cite[Appendix~A, Proposition~A.1.(iv)]{Segal:cohomology}. In $\cU_\bullet$ the space of degenerate simplices is
a topological summand. Therefore $\|\cU \|\lra |\cU|$ is a homotopy equivalence. If we call
the composition $M\slra{\bar{u}} \|\cU\|\lra |\cU|$ again $\bar{u}$ and if 
$\rho\co |\cU|\lra M$ is the canonical projection, then $\rho\circ \bar{u} = \id_M$. If
$\sigma_x = \{j\in J : x\in U_j\}$ then $|\cU|$ can be identified with 
$$
\{(x,t)\in M\times \Delta^J : t\in \Delta^{\sigma_x}\}
$$
but carries a finer topology than the one induced from the product topology.
Nevertheless, $\bar{u}$ is a continuous section of $\rho$ and $\id_{|\cU|}$ is homotopic to
$\bar{u}\circ \rho$ by a homotopy fixed in the first coordinate and linear in the second
coordinate. Continuity of $\bar{u}$ ( also when lifted to $\|\cU\|$ and $\cB\cU$) and of
the homotopy is due to the fact that for each $j$ the closed sets supp$t_j$ and $\partial
U_j$ are disjoint. (See also the proof of Proposition 4.1 in Segal \cite{Segal}.)

Altogether we have:

\begin{thm}\label{thm:hequiv}
Let $(M,\cF)$ be a foliated manifold and let $(U_j,h_j)_{j\in J}$ be a
tangentially categorical cover of $M$ with $J\subset\N$. Let $\cU$ be
the topological category associated to the ordered covering
$(U_j)_{j\in J}$ (with ordering of $J$ induced from $\N$) and $\Pi_T$
the transverse fundamental groupoid of $\cF$ associated to
$(U_j,h_j)_{j\in J}$. Let $\psi\co \cU\ra\Pi_T$ be the obvious
functor associated to these data. Then
$$
\|\psi\|\co \|\cU\| \ra \|\Pi_T\|
$$
is a weak homotopy equivalence.
\end{thm}

\section[The spectral sequence for the singular homology of B Pi sub T]{The spectral sequence for the singular homology of
  $B\Pi_T$}\label{sec:specseq} 

In \cite{Segal} Segal describes a spectral sequence associated to the
(thin) realization of a simplicial space and calculates its
$E_2$--term. Here, we do the same, but for the thick realization and
only for singular homology.

For a simplicial space $X_\bullet$ there is a natural filtration
(natural in $X_\bullet$)
$$
\emptyset = X^{-1} \subset X^0 \subset X^1 \subset\ldots
$$
of the thick realization $\|X\|$ of $X$ where $X^n$ is the image of
$X_n \times \Delta^n$ in $\|X\|$. Because of the presence of
degeneracies each $d_i: X_n \ra X_{n-1}$ is surjective, so that the
image of $X_n \times \Delta^n$ in $\|X\|$ contains the images of $X_i
\times \Delta^i$, $0\leq i\leq n$. Hence, the $\{X^n\}$ form an
ascending sequence. Also, by the usual compactness argument, the $k$--th
singular chain group $S_k (\|X\|)$ of $\|X\|$ is the union of the
chain groups $S_k(X^n)$:
$$
S_k (\|X\|) = \bigcup_{n\geq0} S_k (X^n)
$$
Therefore the homology spectral sequence associated to the filtration
$\{X^n\}$ converges to $H_\ast (\|X\|)$ and has $E^1$--term
$$
E_{r,s}^1 = H_{r+s} (X^r, X^{r-1})
$$
with differential $d^1$ the boundary homomorphism of the
triple $(X^r, X^{r-1}, X^{r-2})$. As in \cite[Proposition~(5.1)]{Segal},
we have: 

\begin{prop}\label{prop:e2term}
The $E^2$--term of the homology spectral sequence associated to the
filtration $\{X^n\}$ of the thick realization $\|X\|$ of the
simplicial space $X_\bullet$ is
$$
E_{r,s}^2 = H_r^\Delta (H_s (X_\bullet))
$$
where $H_\ast^\Delta$ is the simplicial homology of the simplicial
abelian group $H_s (X_\bullet)$.
\end{prop}

\begin{proof}
The proof is easier than in \cite{Segal} since we need not deal with
degenerate simplices, is certainly well known, and probably written up at several places.
Nevertheless, we indicate how to proceed.

Denote by $\dDelta^r$ the boundary of $\Delta^r$, and by
$\Delta_i^{r-1}$ the $i$--th face of $\Delta^r$. Define $\varphi^r: X_r
\times \dDelta^r \ra X^{r-1}$ by mapping $(x_r, (t_0,\ldots,t_{i-1},0,
t_i,\ldots,t_{r-1})) \in X_r \times \smash{\Delta_i^{r-1}}$ to the image of
$(d_i x_r, (t_0,\ldots,t_{r-1})) \in X_{r-1} \times \Delta^{r-1}$ in
$\|X\|$. Because of the simplicial identities between compositions of
face maps, $\varphi^r$ is well defined. Then $X^r$ is obtained from
$X^{r-1}$ by attaching $X_r \times \Delta^r$ along $\varphi^r$. Denote
the``characteristic'' map of the $r$--cells by $\Phi^r\co X_r \times
\Delta^r \ra X^r$.

Then by homotopy invariance and excision 
$$
H_{r+s} (\Phi^r)\co H_{r+s} (X_r \times (\Delta^r,\dDelta^r)) \ra
H_{r+s} (X^r,X^{r-1})
$$
is an isomorphism. Therefore $H_s (X_r)$ is naturally isomorphic to
$$E_{r,s}^1 = H_{r+s} (X^r,X^{r-1}).$$
Under this isomorphism the differential $d^1$ corresponds to the
simplicial boundary map
$$
d_r^\Delta: \Sum_{i=0}^r (-1)^i H_s (d_i): H_s (X_r) \ra H_s
(X_{r-1}) 
$$
of the simplicial abelian group $H_s (X_\bullet)$.

To see this denote the $(r-2)$--skeleton of $\Delta^r$ by $\ddDelta^r$,
and look at the diagram:
$$
\setlength{\unitlength}{.8mm}
\begin{picture}(80,90)(0,-5)
\put(5,3){\makebox(0,0){$H_{k-1}(X_{r-1}\times(\Delta^{r-1},\dDelta^{r-1}))$}}
\put(3,16){\vector(0,-1){10}}
\put(6,12){\makebox(0,0){$D$}}
\put(5,20){\makebox(0,0){$\bigoplus\limits_{i=0}^r H_{k-1}
    (X_r\times(\Delta_i^{r-1},\dDelta_i^{r-1}))$}} 
\put(35,5){\vector(2,1){19}}
\put(49,7){\makebox(0,0){$\Phi_\ast^{r-1}$}}
\put(38,20){\vector(1,0){15}}
\put(75,20){\makebox(0,0){$H_{k-1}(X^{r-1},X^{r-2})$}}
\put(5,40){\makebox(0,0){$H_{k-1}(X_{r}\times(\dDelta^{r},\ddDelta^{r}))$}}
\put(3,25){\vector(0,1){11}}
\put(6,30){\makebox(0,0){$\cong$}}
\put(28,40){\vector(2,-1){27}}
\put(44,36){\makebox(0,0){$\varphi_\ast^{r}$}}
\put(5,55){\makebox(0,0){$H_{k-1}(X_{r}\times\dDelta^{r})$}}
\put(3,51){\vector(0,-1){8}}
\put(24,55){\vector(1,0){36}}
\put(43,58){\makebox(0,0){$\varphi_\ast^{r}$}}
\put(75,55){\makebox(0,0){$H_{k-1}(X^{r-1})$}}
\put(75,51){\vector(0,-1){25}}
\put(5,75){\makebox(0,0){$H_{k}(X_{r}\times(\Delta^{r},\dDelta^{r}))$}}
\put(3,70){\vector(0,-1){10}}
\put(5,65){\makebox(0,0){$\partial$}}
\put(27,75){\vector(1,0){31}}
\put(42,73){\makebox(0,0){$\cong$}}
\put(42,78){\makebox(0,0){$\Phi_\ast^{r}$}}
\put(75,75){\makebox(0,0){$H_{k}(X^{r},X^{r-1})$}}
\put(75,70){\vector(0,-1){10}}
\put(77,65){\makebox(0,0){$\partial$}}
\end{picture}
$$
In this diagram the vertical maps without names are induced by inclusions, and
$D$ restricted to the $i$--th summand is induced by
$(x,(t_0,\ldots,t_{i-1},0,t_i,\ldots,t_{r-1})) \mapsto (d_i
x,(t_0,\ldots,t_{r-1}))$. Clearly, the diagram commutes. The signs of
the simplicial boundary map $d_r^\Delta$ appear when we identify $H_k
(X_r \times (\Delta^r,\dDelta^r))$ with $H_{k-r} (X_r)$ and $H_{k-1}
(X_{r-1} \times (\Delta^{r-1},\dDelta^{r-1}))$ with $H_{k-r}
(X_{r-1})$ via the K\"unneth isomorphism by picking the standard
generators of $H_r (\Delta^r,\dDelta^r)$ and
$H_{r-1}{(\Delta^{r-1},\dDelta^{r-1})}$. 
\end{proof}

Next we show that in a certain range the $E^2$--terms of the spectral
sequences of the simplicial spaces associated to the categories $\cU$
and $\Pi_T$ vanish.

\begin{prop}\label{prop:hommanifold}
Let $Y$ be a not necessarily Hausdorff $k$--manifold. Then $H_i(Y)=0$
for $i>k$.
\end{prop}

\begin{proof}
Since $H_i$ commutes with colimits and by induction (if $Y$ is not second countable, use
transfinite induction)  it suffices to prove the following. Let $Y=Z\cup W$ with $Z,W$
open in
$Y$ and
$W$ homeomorphic to an open subset of $\R^k$ and assume that the
proposition holds for the $k$--manifold $Z$. Then it also holds for
$Y$. But this easily follows from the Mayer--Vietoris sequence which we
may use since the triple $\{Y;Z,W\}$ is excisive for singular
homology.
\end{proof}

\begin{rem}
Initially we intended to use \v{C}ech cohomology instead of singular
homology since \v{C}ech cohomology has better properties with respect to
dimension. But in \v{C}ech cohomology a triple $\{Y;Z,W\}$ with $Z,W$
open in $Y$ need not be excisive and the corresponding Mayer--Vietoris
sequence need not be exact.
\end{rem}

\begin{cor}\label{cor:e2fundgroup}
Let $\Pi_T$ be the transverse fundamental groupoid associated to a
tangentially categorical cover of a manifold with foliation of
codimension $k$. Let $E_{r,s}^2 (\Pi_T)$ be the $E^2$--term of the
spectral sequence for the thick realization $\|\Pi_T\|$ of the
simplicial space associated to the topological category $\Pi_T$. Then
$$
E_{r,s}^2 (\Pi_T) = 0 \quad\textrm{ for } s>k\,.
$$
\end{cor}

\begin{proof}
This follows immediately from Propositions \ref{prop:e2term},
\ref{prop:composables} and \ref{prop:hommanifold}.
\end{proof}

Let $\dU = (U_j)_{j\in J}$, $J\subset\N$, be an ordered open cover of
a topological space. Let $\cU\subset\Gamma_\dU$ be the category
associated to this ordered covering and $\cU_\bullet$ the simplicial
space which in turn is associated to this category. If the dimension
of the nerve of $\dU$ is $k$ then for all $r>k$ all $r$--simplices of
$\cU_\bullet$ are degenerate. This implies that for any functor $h$ from
the category of topological spaces to the category of abelian groups
every $r$--chain of the chain complex associated to the simplicial
abelian group $h\cU_\bullet$ is degenerate if $r>k$. Therefore we obtain
from \fullref{prop:e2term}

\begin{cor}\label{cor:e2cover}
Let $\dU = (U_j)_{j\in J}$, $J\subset\N$, be  an ordered open cover of
dimension $k$ of the topological space $X$. Then for the $E^2$--term
$E_{r,s}^2(\cU)$ of the homology spectral sequence associated to the
thick realization of the simplicial space $\cU_\bullet$ associated to
the category $\cU\subset\Gamma_\dU$ associated to the ordered covering
$\dU$ we have that
$$
E_{r,s}^2 (\cU) = 0
$$
for $r>k$.
\end{cor}

\begin{proof}[Proof of \fullref{main}]
Let $\cF$ be a $p$--dimensional $K(\pi,1)$--foliation of the $n$--manifold $M$ .
Let 
$\{(U_0,h_0),\ldots,(U_k,h_k)\}$ be a tangentially categorical cover of $M$. Let
$\cU\subset\Gamma_\dU$ be the topological category associated to the
ordered cover $\dU = (U_0,\ldots,U_k)$, let  $\Pi_T$ be the associated
transverse fundamental groupoid and let $\psi\co \cU\ra\Pi_T$ be the restriction
to $\cU$ of the groupoid homomorphism $\psi\co \Gamma_\dU \ra \Pi_T$
associated to these data (see \fullref{sec:tgroup}). By
\fullref{thm:hequiv} $\psi$ induces a weak homotopy equivalence
between the associated ``thick'' classifying spaces, ie,
$$
\|\psi\|\co \|\cU\| \ra \|\Pi_T\|
$$
is a weak homotopy equivalence. Since the spectral sequences
associated to the thick realizations of simplicial spaces converge 
all elements of $H_t (\|\cU\|)$ lie in filtration $r\leq k$ by
\fullref{cor:e2cover} while all elements of filtration $r< t-n+p$ of
$H_t (\|\Pi_T\|)$ vanish by  \fullref{cor:e2fundgroup}. 
Since $\|\cU\|$ and $M$ are homotopy equivalent, $H_t (\|\cU\|)\neq0$ by
assumption. Furthermore $\|\psi\|$ is a filtration preserving  weak homotopy
equivalence. Therefore, $k\geq  t-n+p$, ie $\cat\cF\geq t-(n-p)$, as claimed.
\end{proof}

\section[The failure of foliated cohomological dimension as a lower
  bound for tangential category of K(pi,1)-foliations]{The failure of foliated cohomological dimension as a lower
  bound for tangential category of $K(\pi,1)$--foliations}
\label{section:folcoh}

While the lower bound for $\cat\cF$ of \fullref{main} is exact for
$K(\pi,1)$--foliations $\cF$ of closed manifolds (\fullref{cor:main}) it
falls in general way short of the mark for $K(\pi,1)$--foliations of positive
codimension on noncompact manifolds. This is in contrast to the main
result of \cite{EilGan} where apart from groups $\pi$ of cohomological
dimension~2 the lower bound $\dim\pi$ equals the LS--category of a
$K(\pi,1)$. 

If we denote, in analogy with the definition for groups, by
$\dim\Gamma$ the homological dimension of a classifying space of the
groupoid $\Gamma$, then our bound for $\cat\cF$ is
$$
\dim\Pi_T - \codim\cF \leq \cat\cF
$$
where $\Pi_T$ is the transverse groupoid of the $K(\pi,1)$--foliation
$\cF$ associated to some tangentially categorical cover.

It is easy to pass from a $K(\pi,1)$--foliation $\cF$ to a new
$K(\pi,1)$--foliation $\cF'$ without changing the tangential category
and $\dim\Pi_T$ but increasing the codimension. Simply multiply the
manifold $M$ on which $\cF$ is defined by $\R^k$ and foliate
$M\times\R^k$ by $L\times\{y\}$, $L\in\cF$, $y\in\R^k$. Consequently
our lower bound can miss the target $\cat\cF$ substantially. Another
example is the foliation $\dR_x$ on $S^3 \smetmi \{x\}$ obtained from
the Reeb foliation $\dR$ of $S^3$ by removing the point $x\in S^3$. If
$x$ does not lie on the toral leaf of $\dR$ then $\cat\dR_x =2$. Our
lower bound for $\cat\dR_x$  is
$-1$ and thus utterly useless.

Looking for an invariant less dependent on the codimension and
generalizing $\dim\pi$ for a $K(\pi,1)$--foliation with a single leaf,
foliated or leafwise cohomology comes to mind. Its dimension will not
change when multiplying the total space (but not the leaves) of a
foliated manifold by another manifold. Also for $\dR_x$ we have $H^2 (\dR_x)\neq0$, and
thus the foliated cohomological dimension of $\dR_x$ equals $\cat\dR_x$ if $x$ is not
on the toral leaf.

As mentioned earlier foliated cohomological dimension has its draw backs when estimating
$\cat\cF$ on closed manifolds. By Proposition 6.2 of
\cite{ColHurTan} the second foliated cohomology group $H^2 (\cF)=0$ for the weakly
unstable foliation $\cF$ of the geodesic flow of any closed surface of constant
negative curvature. Since $\cF$ is a $2$--dimensional
$K(\pi,1)$--foliation on a closed $3$--manifold, $\cat\cF=2$ by our
estimate. (For this example see also Proposition~6.4 of
\cite{ColHurTan}.) So in this case 
foliated cohomological dimension of $\cF <\cat\cF$.

But worse, in general the foliated cohomological dimension is not a
lower bound for $\cat\cF$ of $K(\pi,1)$--foliations $\cF$ on open
manifolds. 

In fact, in many cases the foliation $\cF_K$ obtained from a
$p$--dimensional foliation $\cF$ of a manifold $M$ by restriction to
$M\smetmi K$ for some closed subset $K$ of $M$ has the property that
$H^p (\cF_K)\neq0$ (even if $H^p (\cF)=0$), see
\fullref{prop:folpct} below. At the end of this section we give a simple example
of a $K(\pi,1)$--foliation where this happens but where $\cat\cF<p$. Further examples are
the Reeb foliations punctured at the toral leaf. This will be shown in the next section.
So foliated cohomological dimension does not qualify as a lower bound for
$\cat\cF$ for
$K(\pi,1)$--foliations. 

For the remaining part of this section we assume that all manifolds
are smooth and that all foliations are leafwise smooth, ie, the
leaves are smoothly immersed into the manifold. Transversely the
foliation is $C^r$ for some $0\leq r\leq\infty$ and foliated
cohomology is defined via forms which are leafwise smooth and which are together with all
their leafwise derivatives transversely $C^r$.

\begin{prop}\label{prop:folpct}
Let $\cF$ be a $p$--dimensional foliation of the manifold $M$ and let
$K\subset M$ be a closed subset such that there exists a foliation chart
neighborhood for $\cF$ which we identify with $\R^p \times \R^{n-p}$
with the standard $p$--dimensional foliation so that the following
conditions hold:
\begin{itemize}
\item[\rm(i)]  $(0,0)\in K$.
\item[\rm(ii)]  There exists a compact neighborhood $U$ of $0$ in $\R^p$
  with smooth boundary $\partial U$ such that $\partial U \times \{0\}
  \cap K=\varnothing$.
\item[\rm(iii)] There exists a sequence $(y_i)$ in $\R^{n-p}$ converging
  to $0$ such that $U\times \{y_i\} \cap K=\varnothing$ for all $i$.
\end{itemize}
Then $\dim H^p(\cF_K)=\infty$, where $\cF_K$ is the
foliation induced by $\cF$ on $M\setminus K$.
\end{prop}

\begin{proof}
Denote $\{x\in\R^q : \|x\| \leq t\}$ by $B_t^q$. Let $\omega$ be a
leafwise
$p$--form on $M\setminus\{(0,0)\}$ which vanishes outside the subset
$\smash{B_2^p \times B_2^{n-p}}$ of the chart neighborhood $\R^p \times
\R^{n-p} \subset M$ given by the hypotheses of the proposition and
which on $B_1^p \times B_1^{n-p} \smetmi \{(0,0)\}$ is defined as
follows.

Choose $1>a>0$ such that $B_a^p \subset \circU$ and let $\varphi\co \R^p
\ra [0,1]$ be smooth with support in $\smash{B_a^p}$ such that
$$
\int\limits_{\R^p} \varphi\, dx^1 \wedge \ldots \wedge dx^p = b > 0\,.
$$
Then we set for $y\in B_1^{n-p} \smetmi \{0\}$, $x\in\R^p$, $c>0$ 
$$
\omega(x,y) = \omega_c (x,y) = \frac{1}{\|y\|^{p+c}} \;\varphi
\Big(\frac{x}{\|y\|} \Big)\, dx^1 \wedge \ldots \wedge dx^p
$$
and $\omega(x,0)=0$ for $x\neq0$. Clearly, with respect to the usual
smooth structure on $\R^p \times \R^{n-p}$ the form $\omega_c$ is
smooth on $B_1^p \times B_1^{n-p} \smetmi \{(0,0)\}$ and thus is
a leafwise $p$--form for $\cF\mid_{M\smetmi\{(0,0)\}}$ which is smooth in
the leaf direction and transversely $C^r$. 

Assume that $\omega_c$ is exact in the leafwise deRham complex of
$\smash{\cF_K = \cF\mid_{M\smetmi K}}$. Let $\eta$ be a leafwise $(p-1)$--form
with $d_{\cF_K} \eta = \omega_c$. Consider the sequence $(y_i) \ra 0$
in $\R^{n-p}$ given by the hypotheses of the proposition. We may
assume that all $y_i \in B_1^{n-p}$. Since $U\times\{y_i\} \cap K =
\varnothing$ we have
$$
\int\limits_{\partial U\times\{y_i\}} \eta =
\int\limits_{U\times\{y_i\}} \omega_c = \frac{1}{\|y_i\|^c}
\int\limits_{\R^p} \varphi\, dx^1 \wedge\ldots\wedge dx^p =
\frac{b}{\|y_i\|^c}. 
$$
In particular $\lim_i \int_{\partial U\times\{y_i\}}
\eta$ is infinite. But $\partial U\times \{0\} \cap K$ is empty.
 Therefore, $\lim_i \int_{\partial
  U\times\{y_i\}} = \int_{\partial U\times\{0\}} \eta$ which is
finite. 

The same argument shows that the family $\{[\omega_c]: c>0\}$ of
leafwise $p$--dimensional cohomology classes is $\R$--linearly
independent. 
\end{proof}
\begin{example} Let $(x,y,z)$ be the standard coordinates of $\R^3$. Foliate $\R^3$ by
the horizontal planes of $\R^3$, ie, the planes parallel to the $xy$--plane. Let
$M=\R^3\smetmi\{0\}$, and let $\cF$ be the induced foliation on $M$. By
\fullref{prop:folpct} $H^2(\cF)\neq 0$. But clearly, $\cat\cF=1$. To see this note
the following. $\R^2\smetmi\{0\}$ is a leaf of $\cF$. Therefore, $\cat\cF>0$. The sets
$U_{-1}:=\R^3\smetmi\{(x,0,0)|x\geq 0\}$ and $U_1:=\R^3\smetmi\{(x,0,0)|x\leq 0\}$ form a
tangentially categorical cover of $(M,\cF)$. For $\epsilon = -1, 1$ the tangential
contractions  move the point $(x,y,z) \in U_\epsilon$ with constant speed on a straight
line to
$(\epsilon,0,z)$.  Therefore, $\cat\cF\leq 1$. 
\end{example}
\begin{rem} 
The proof of \fullref{prop:folpct} benefits from the fact that forms might be
unbounded at infinity. But foliated cohomology with compact support gives the wrong
estimate for foliations with a single contractible leaf, while foliated bounded
cohomology for foliations with a single leaf with amenable fundamental group will give
very poor estimates. So some new idea is needed.
\end{rem}
 \begin{rem} In \cite{Shul} and \cite{ColHurTan} nonvanishing secondary classes of
$\cF$  are used to provide lower bounds for the number of foliation charts needed for 
$\cF$ (\cite{Shul})  or for $\cat\cF$   \cite{ColHurTan}. The proofs make use of the
de Rham complex of the simplicial manifold $\Gamma^q_\bullet$, where$\Gamma^q$ is the
groupoid of germs of local diffeomorphisms of $\R^q$. A nonvanishing secondary class
implies a nonvanishing de Rham class in the total complex of the de Rham double complex
associated to the simplicial manifold $\Gamma^q_\bullet$, and from this fact the
estimates are obtained. This approach
will not work in our situation where we compare the homology of the manifold $M$ with the
homology of the transverse fundamental groupoid $\|\Pi_T\|$. In order to use the de Rham
complex we need the map from the cohomology of the de Rham double complex to the real
singular homology of
$\|\Pi_T\|$ to be surjective. But this is not always the case. A  simple calculation for
the foliation of our example above will show that in this case the cohomology of the de
Rham double complex of  ${\Pi_T}_\bullet$ is trivial in
positive degrees, while $H^2(\|\Pi_T\|;\R)=H^2(\R^3\smetmi\{0\};\R)=\R.$ 
\end{rem}

\section{Tangential LS--category of finitely punctured Reeb foliations}
\label{section:Reeb}

A Reeb foliation of the solid torus $D^2\times S^1 = D^2 \times \R/2\Z$
is given by a smooth even function $f\co (-1,1) \lra [0,\infty)$ with
$f(0)=0$ and $f|_{(0,1)}\co (0,1) \lra (0,\infty)$ a diffeomorphism. The
leaves of the foliation are the images of the graphs of $\bar{f}_a\co
\circD \lra [a,\infty)$, $\bar{f}_a (x) = f(\|x\|)+a, a\in\R$, 
under the covering map $\circD \times\R \lra
\circD \times S^1$ together with  the boundary $\partial D^2 
\times
S^1$ of $D^2 \times S^1$ as the only compact leaf. All Reeb foliations
of $D^2 \times S^1$ are homeomorphic via a foliation preserving
homeomorphism which restricts to the identity on $\partial D^2 \times
S^1$. Given coprime integers $p,q$, let $L(p,q)$ be the usual
$3$--dimensional $(p,q)$--lens space. Here $S^3 = L(1,0)$ and $S^2 \times
S^1 = L(0,1)$ are included among the lens spaces. Being the union of
two solid tori which intersect in their common boundary, $L(p,q)$
carries a natural foliation coming from the Reeb foliation of the
solid tori. We denote this foliation by $\cR(p,q)$. If $E\subset L(p,q)$ is a closed
subset, we denote by $\cR_E(p,q)$  the restriction of
$\cR(p,q)$ to $L(p,q)\smetmi E$. In this section we prove:

\begin{prop}\label{prop:Reebpct} Let $E\subset L(p,q)$ be finite, and let
$T\subset L(p,q)$ be the toral leaf of $\cR(p,q)$.  
\begin{itemize}
\item[\rm(i)] $\cat\cR_E (p,q) = 2$, if $E\cap T=\varnothing$ or $E\cap
  (L(p,q)\smetmi T) \neq \varnothing$.
\item[\rm(ii)] $\cat\cR_E (1,0) =\cat\cR_E (0,1)=1$, if $\varnothing\neq E\subset
T$. 
\end{itemize}
\end{prop}

\begin{rem}
Thus, by \fullref{prop:folpct}, $\cR_E (1,0)$ and $\cR_E (1,0)$ with
$\varnothing\neq E\subset
T$, $E$  finite, provide further examples where foliated
cohomological dimension is not a lower bound for tangential category.
\end{rem}

\begin{proof}
[Proof of \fullref{prop:Reebpct}(i)] By \cite[Theorem 5.2]{SinV}, we have
$\cat\cF \leq \dim\cF$ for any $C^2$--foliation $\cF$. Since $\cR(p,q)$
is homeomorphic to a $C^\infty$--foliation, $\cat\cR_E (p,q) \leq2$ for
any $p,q,E$. On the other hand by \cite[Proposici\'{o}n~4.10]{Colthesis}, 
the usual category of a leaf of a foliation $\cF$ is a lower
bound for $\cat\cF$. Thus $\cat\cR_E (p,q) =2$, if $E\cap T
=\varnothing$. If $E\cap (L(p,q) \smetmi T) \neq\varnothing$ then
$\cR_E (p,q)$ contains a  leaf with at least two ends, and with
a simple end accumulating to $T\smetmi E$ according to 
\fullref{def:accumend} below. Thus $\cat\cR_E (p,q) =2$ is a special
case of \fullref{prop:accum} below.
\end{proof}

An end of an $n$--manifold $V$ is an element of
$\varprojlim_{K} \pi_0 (V\smetmi K)$, $K\subset V$
compact. Instead of all compact subsets of $V$ it suffices to consider
a sequence $K_1 \subset K_2 \subset\cdots$ such that each $K_i$ is a
compact sub--$n$--manifold with boundary, $K_i \subset \text{int\,} K_{i+1}$ for
all $i$, no component of $V\smetmi \text{int\,} K_i$ is compact, $V$ is the
union of the $K_i$, and such that each component of $V\smetmi \text{int\,} K_i$
intersects exactly one component of $\partial K_i$. An end $e$ of $V$ is then a
sequence
$C_1
\supset C_2 \supset C_3 \supset\cdots$ where each $C_i$ is a component of
$V\smetmi \text{int\,} K_i$. A subset $W$ of $V$ accumulates to $e = (C_i)$, if $W\cap C_i
\neq \emptyset$ for all $i$. An end $C_1 \supset C_2 \supset C_3 \supset\cdots$ is called
\textit{simple} if $\partial(K_{i+1}\cap C_{i})$ has exactly two components for large
enough
$i$. These components are then $\partial C_i$ and $\partial C_{i+1}$.

\begin{defn}\label{def:accumend}
Let $\cF$ be a foliation of $M$, $A\subset M$ a subset and $e$ an end
of a leaf $L$ of $\cF$. We say that $e$ accumulates to $A$, if every connected
subset of $L$ which accumulates to $e$  contains a sequence of points converging
to a point of $A$.
\end{defn}
\begin{prop}\label{prop:accum} Let $L$ be a leaf of a $p$--dimensional
$C^1$--foliation $\cF$. Assume that $p\geq 2$, that $L$ has at least two ends, and that $L$
has a simple end which accumulates to a leaf $L'$ of $\cF$ different from $L$. Then
$\cat\cF \geq 2$.
\end{prop}
\begin{proof}  Let
$U_0,U_1$ be a tangentially categorical open cover of the foliated manifold. Then
$U_0\cap L, U_1\cap L$ is a categorical cover of $L$ in the usual sense where $L$ is
given the leaf topology (see the proof of Proposici\'on 4.10 in
\cite{Colthesis}). Let $e=C_1\supset C_2\supset \cdots $
be a simple end  of $L$  in the notation introduced above which accumulates to the leaf $L'\neq L$.
By definition there exists $i_0$ such that for $i\geq i_0$ we have
$$
\partial(C_i\cap K_{i+1})=\partial C_i \sqcup \partial C_{i+1}.
$$ 
Let $W\subset L$ be a $p$--dimensional submanifold with boundary, closed as a subset of
$L$ (with the leaf topology), with $W\subset U_0\cap L$ and $L\smetmi \text{int\,} W
\subset U_1\cap L$, and with $\partial W$ transverse to $\partial C_i$ for all $i\geq
i_0$.

It is straight forward to see that $W$ with the required properties exists: since L is normal we find open subsets $X,Y$ in $L$ such that $L\setminus U_1\subset X\subset \bar{X}\subset Y\subset \bar{Y}\subset U_0$ and a continuous function $f_0\co L\longrightarrow [0,1]$ with $f(\bar{X})=\{0\}, f(L\setminus  Y)=\{1\}$; let $0<t<1$ be a regular value for a smooth approximation $f\co L\longrightarrow [0,1]$ of $f_0$ which is equal to $f_0$ in a neighborhood of $(L\setminus U_1)\cup (L\setminus U_0)$; if necessary, we change $f$ by a small smooth isotopy of $L$ to make sure that $f^{-1}(a)$ is transverse to $\partial C_i$ for all $i\geq
i_0$; then $W:=f^{-1}([0,a]$ satisfies all requirements. Notice also the following. If $h\co U_0\times I\lra M$ is a tangential homotopy contracting the leaves of the foliation induced on $U_0$ to points, then $h_1$ maps every component of $W$ to a point since any such component is contained in a component of $L\cap U_0$, and the components of $L\cap U_0$ are leaves of the induced foliation.

Since $\partial C_i, i\geq i_0,$ does not bound in $L$, and since $W, L\smetmi
\text{int\,} W$ are contractible in $L$, the compact manifold $W\cap K_j\cap C_i$ must
intersect both components of $\partial(K_j\cap C_i)=\partial C_j\sqcup \partial C_i$ for
$j>i\geq i_0$. If every component of $W\cap K_j\cap C_i$ intersects at most one of
$\partial C_j$ and $\partial C_i$ then we find a closed $(p-1)$--manifold $S\subset
K_j\cap C_i$ separating $\partial C_j$ from $\partial C_i$ with $S\cap W=\emptyset$.
Since $S$ separates $\partial C_j$ from $\partial C_i$ it cannot bound in $L$. Since
$S\subset U_1\cap L$ it bounds in $L$, and so we get a contradiction.   Since for each
$j>i\geq i_0$ the manifold $W\cap K_j\cap C_i$ has only finitely many components, there
exists a component $W_0$ of $W$ which intersects every $\partial C_i, i\geq i_0$, and
therefore accumulates to $e$. Since $e$ accumulates to $L'$ and $W_0$ is connected, we find  a sequence
$(x_i)$ in
$W_0$ converging to a point
$x$ in a leaf
$L'$. Let $h\co  U_0\times I \lra  M$ be a tangential homotopy
contracting the leaves of the foliation induced on 
${U_0}$ to points. Then, as we noticed above, $h_1(W_0)$ is a point $y$ of $L$. Since $(x_i)$
converges to $x$ we have $h_1(x)=y\in L$. Since $h$ is  a tangential homotopy, $h_1(x)\in
L'$. Since $L\neq L'$,  we obtain a contradiction.
\end{proof}

The remaining part of this section will be concerned with the proof of
\fullref{prop:Reebpct}(ii).

In the course of the proof we will come across the images of foliated open sets after a partial tangential contraction where the notion of foliation does not apply any more. Rather, the images of the leaves form a partition of the resulting set into connected subspaces. To deal with this situation
we use the following notation. If $\cP$ is a partition of a topological
space $X$ into connected subsets a homotopy $f\co U\times I\ra X$ is called a 
\textit{$\cP$--homotopy} if all paths $f_u \co I\ra X$, $f_u(t) := f(u,t)$,
$u\in U$, lie in a set of the partition $\cP$. A $\cP$--homotopy $f$ is called a 
\textit{$\cP$--deformation}, if  $f(u,0)=u$ for all $u\in
U$. If all paths $f_u, u\in U$, of a $\cP$--deformation lie in a set
$U'$ we say that $f$ is a \textit{$\cP$--deformation inside $U'$}. A
$\cP$--deformation $f\co U\times I\ra X$ is called a 
\textit{$\cP$--contraction} if for every $P\in\cP$ the restriction of
$f(-,1)\co U\ra X$ to every component of $P\cap U$ is constant.

We will tacitly assume that all partitions considered are  partitions by
\textit{connected} sets. Thus the partition of a subspace $U$ of $X$ induced
from a partition $\cP$ of $X$ consists of the components of $P\cap U$,
$P\in\cP$. Therefore, a $\cP$--deformation $f\co U\times I\ra X$ is a
$\cP$--contraction if $f(-,1)$ restricted to any element of the partition
induced from $\cP$ is constant.

It will be convenient to use the following model for the Reeb
foliation of $D^2 \times \R$ and $D^2 \times S^1$. Let $H=\C\times [0,\infty) \smetmi
\{(0,0)\}$. Foliate $H$ by the horizontal planes $\C\times\{t\}$,
$t>0$, and $(\C\smetmi\{0\}) \times \{0\}$. Denote this foliation by
$\cP$. We identify $H$ in the obvious way with a subspace of $\R^3 =
\C\times\R$. Let $\smash{S_+^2} = \{(z,t)\in H: |z|^2 +t^2 =1\}$ be the upper
hemisphere of the unit sphere $S^2$ of $\R^3$ and let $\sigma \co S^2
\smetmi \{(0,-1)\} \lra \R^2$ be the stereographic projection from the
south pole $(0,-1)$. Then
$$
t\cdot x \longmapsto (\sigma(x),\log t)\,,\; x\in S^2_+ \,,\; t>0\,,
$$
defines a diffeomorphism $\Sigma\co H\lra D^2 \times\R$. On $H$ we let
$\R$ act by $s(t\cdot x) = e^s \cdot t\cdot x$, $s\in\R$, $x\in
S^2_+$, $t>0$, and on $D^2 \times\R$ by translation on the second
factor. Then $\Sigma$ is equivariant and induces a diffeomorphism from
$H\big/2\Z$ to the solid torus $D^2 \times\R\big/2\Z$. Since the
foliation $\cP$ is preserved by the action the solid torus $H\big/2\Z$
inherits a foliation denoted by $\cQ$. This is our model of the Reeb
foliation. All our tangential homotopies in $(H\big/2\Z, \cQ)$ will be
defined on subsets which lift diffeomorphically to fundamental
domains of the covering $H \lra H\big/2\Z$. We will use the following:

\medskip{\bf Notations for subsets of $H$}\qua
\begin{align*}&\hbox{$H^+ := \{(z,t)\in H: t>0\}$.}\\
&\hbox{For $F\subset \C\smetmi\{0\}$ set $H_F = H^+ \cup (F\times\{0\})$.}
\\
&\hbox{For $0<a<\infty$ set $H(a) = \{(z,t)\in H: a^2 e^{-2} < |z|^2
+t^2 < a^2 e^2 \}$ (see \fullref{figure}).}
\\
&\hbox{For any subset $A\subset H$ set $\partial A = \{(z,t)\in A: t=0\}$.}
\end{align*}

\begin{lem}\label{lem:cofcontr}
Let $F \subset \C \smetmi\{0\}$, let $h\co  F\times I \lra \C\smetmi
\{0\}$ be a contraction, and let $F\subset\C$ have the
homotopy extension property with respect to $\C$. Then there exists a
$\cP$--contraction of $H_F$ inside $H_{h(F\times I)}$.
\end{lem}

\begin{proof}
Obvious.
\end{proof}

\begin{figure}[ht!]\centering
\setlength{\unitlength}{1mm}
\begin{picture}(0,80)(5,0)\small
\put(0,40){\makebox(0,0){\includegraphics[scale=0.80]{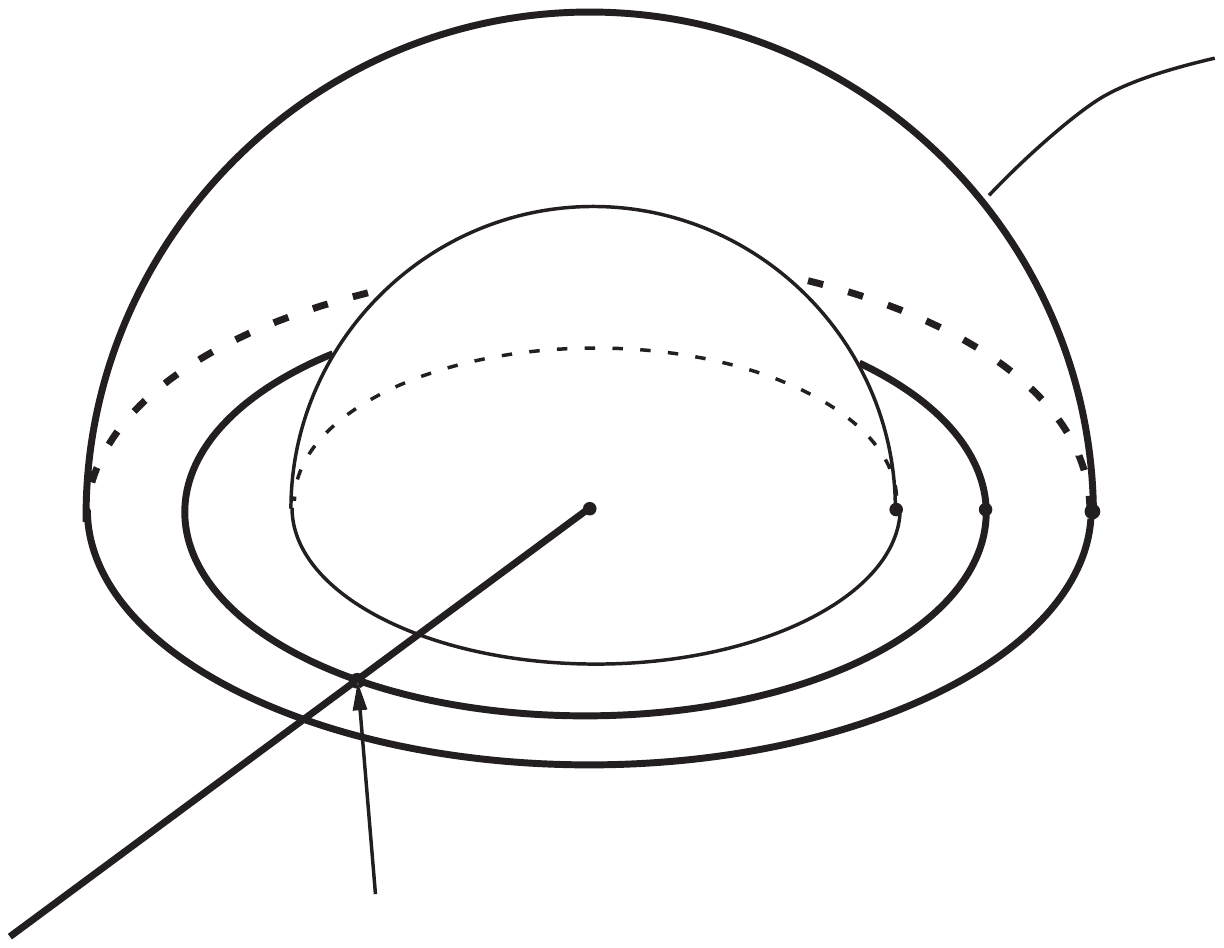}}}
\put(-1,35){\makebox(0,0){$0$}}
\put(18.5,37){\makebox(0,0){$ae^{-1}$}}
\put(33,37){\makebox(0,0){$a$}}
\put(44,37){\makebox(0,0){$a\cdot e$}}
\put(55,74){\makebox(0,0){$D(a\cdot e)$}}
\put(-39,5){\makebox(0,0){$L_b$}}
\put(-15,5){\makebox(0,0){$a\cdot e^{i\pi b}$}}
\put(27,5){\makebox(0,0){$H(a)$}}
\end{picture}
\caption{}\label{figure}
\end{figure}
  
\begin{lem}\label{lem:round}
For each $0<a<\infty$ there exists a $\cP$--deformation of $H(a)$
inside $H(a)$ which on $\partial H(a)$ is given by
$$
(a\cdot e^{x+\pi iy},0,s) \longmapsto (a\cdot e^{(1-s)\cdot x+\pi
  iy},0) \,,
$$
$-1<x<1$, $y\in\R\big/2\Z$, $s\in[0,1]$.
\end{lem}

\begin{proof}
This should be clear. Points $(z,t) \in H(a)$ with $t\geq a\cdot
e^{-1}$ will not be moved. For $0\leq t\leq a\cdot e^{-t}$ points of
$H(a)$ with second coordinate equal to $t$ form an annulus
$\{(z,t)\in H: r(t) < |z| < R(t)\}$ with $r(t) = (a^2 \cdot e^{-2}
-t^2)^{\frac{1}{2}}$, $R(t) = (a^2 e^2 -t^2)^{\frac{1}{2}}$. Choose
continuous $S,s\co [0,a\cdot e^{-1}] \lra [0,\infty)$ such that
$s(a\cdot e^{-1})=0$, $r(t) \leq s(t) \leq a$, $s(0)=a$; $S(a\cdot
e^{-1}) = R(a\cdot e^{-1})$, $a\leq S(t) \leq R(t)$, $S(0)=a$. Then
push points of $H(a)$ with second coordinate $t$ radially with
constant speed to the annulus $\{(z,t)\in H(a): s(t) \leq |z| \leq
S(t)\}$, the speed depending on the distance from this annulus.
\end{proof}

For $b\in\R\big/2\Z$ let $L_b$ be the ray $\{a\cdot e^{\pi ib}:
0<a<\infty\}$ in $\C$ (see \fullref{figure}). The next Lemma is again 
straightforward. 

\begin{lem}\label{lem:rip}
Let $F = \C\smetmi (\{0\}\cup L_1)$. Then there exists a
$\cP$--deformation of $H_F$ in $H_F$ which on $F=\partial H_F$ is given
by
$$
(e^{x+\pi iy},0,s) \longmapsto (e^{x+\pi i(1-s)y},0)\,,
$$
$x\in\R$, $y\in (-1,1)$, $s\in [0,1]$.
\end{lem}

\begin{proof}
For $t>0$ our $\cP$--deformation (with deformation parameter $s$) is of
the form
$$
(e^{x+i\pi y},t,s) \longmapsto (e^{x+i\pi f(y,t,s)},t)\,,
$$
for $-1\leq y\leq 1$, $s\in[0,1]$. For $t\geq1$ we let $f(y,t,s)=y$,
and for all $t>0$, $s\in[0,1]$, $y\in[-1,1]$ we let $f(-y,t,s) =
-f(y,t,s)$. Furthermore, for $0<t<1$ and $s\in[0,1]$ the map
$f(-,t,s)\co [-1,1] \lra\R$ is linear on $[-1,t-1]$ and $[t-1,0]$ and
maps $-1$ to $-1$, $t-1$ to $(1-s)(t-1)$, and $0$ to $0$. As $t$ goes
to $0$ this map converges for $-1<y<1$ to the desired homotopy on
$F=\partial H_F$.
\end{proof}

Both, $L(1,0)$ and $L(0,1)$, are the union of two copies of $H\big/2\Z
=: V$. These copies will be denoted by $V_1$ and $V_2$ and their universal
coverings by $H_1$ and $H_2$. Also for any subset $X$ of $V$ (or $H$)
we will denote the corresponding set in $V_i$ (or $H_i$) by $X_i$. If
the projection $\rho\co H\lra V$ maps $X\subset H$ diffeomorphically to
its image, we will often denote its image also by $X$.

The standard meridional disks of $V$ are the images of the disks
$$
D(a) = \{(z,t)\in H: |z|^2 + t^2 = a^2 \}
$$
(see \fullref{figure}), and the standard parallels of $\partial V$ are the
images of the rays 
$L_b$, $b\in\R\big/2\Z$. The image of $L_b$ in $\partial V$ is denoted
by $\lambda(b)$. We obtain $L(0,1) = S^2 \times S^1$ by attaching
$V_2$ to $V_1$ along the ``identity'' map $\partial V_2 \lra \partial
V_1$, ie, $\smash{(e^{x+i\pi y},0)_2}$ and $\smash{(e^{x+i\pi y},0)_1}$ are equal in $L(0,1)$, while the
attaching map
$\partial V_2 \lra \partial V_1$ for $L(1,0) = S^3$ identifies
$(e^{x+i\pi y},0)_2$ and $(e^{y+i\pi x},0)_1$ in $L(1,0)$.

We may assume that the finite set $E$ of \fullref{prop:Reebpct} is contained in the meridian of $V_1$ which bounds
the disk $D(e)_1$.

\begin{prop}\label{prop:thick}
For $(p,q) = (0,1)$ or $(p,q) = (1,0)$ the set $L(p,q)\smetmi (D(e)_1
\cup D(e)_2 \cup \lambda(1)_1)$ is $\cR_E (p,q)$--categorical.
\end{prop}

\begin{rem}
Note that for $(p,q) = (1,0)$ the set $\lambda(1)_1$ equals $\partial
D(e)_2$, while for $(p,q) = (0,1)$ we have $(\lambda(1)_1) \cap
(\partial D(e)_2) = (-e,0)_1 = (-e,0)_2$ in $L(0,1)$.
\end{rem}

\begin{proof}
The cylinders $V_i \smetmi D(e)_i$, $i=1,2$, lift diffeomorphically to
$H(1)_i$, and $V_i \smetmi (D(e)_1 \cup D(e)_2 \cup \lambda(1)_1)$
lifts diffeomorphically to $H(1)_i \smetmi L_{1i}$.

On $H(1)_1 \smetmi L_{11}$ we first use the $\cP$--deformation of \fullref{lem:round} for $a=1$. After this $\cP$--deformation we use the
$\cP$--deformation of \fullref{lem:rip}. Since this $\cP$--deformation
when restricted to $H(1)_1 \smetmi L_{11}$ is a deformation inside
$H(1)_1 \smetmi L_{11}$ this defines an $\cR_E(p,q)$--deformation on
its image in $L_E(p,q)$, $(p,q) = (1,0)$ or $(0,1)$.

For $(p,q) = (0,1)$ we do the same $\cP$--deformations in the same
order on $H(1)_2 \smetmi L_{12}$. These agree in $L_E (0,1)$ on the
common boundary $\partial (H(1)_1 \smetmi L_{11}) = \partial (H(1)_2
\smetmi L_{12})$.

For $(p,q) = (1,0)$, ie on $S^3$, we do these $\cP$--deformations on
$H(1)_2 \smetmi L_{12}$ in reverse order. Then they will again agree
on their common boundary in $L(1,0)$. After these $\cP$--deformations
the image of $\partial H(1)_1 \smetmi L_{11}$ ($=$ image of $\partial
H(1)_2 \smetmi L_{12}$) will consist of the single point $(1,0)_1 =
(1,0)_2$ in $L(p,q)$. This means that the image of $H(1)_i \smetmi
L_{1i}$ is contained in $H_{Fi}$, with $F=(1,0) \in H$. So we may
apply the $\cP$--contraction of \fullref{lem:cofcontr} with $h\co 
F\times I \lra \C\smetmi\{0\}$ the constant homotopy to these
images. They agree on their common point of intersection $(1,0)_1 =
(1,0)_2$ and thus project down to an $\cR_E (p,q)$--deformation on $L_E
(p,q)$. Altogether we obtain an $\cR_E (p,q)$--contraction of $L(p,q)
\smetmi (D(e)_1 \cup D(E)_2 \cup \lambda(1)_1)$.
\end{proof}

Up till now we have only assumed that $E\subset \partial D(e)_1$. By
hypothesis $E\not=\varnothing$. So we may further assume that
$(-e,0)_1 = (-e,0)_2 \in \partial D(e)_1 \cap \partial D(e)_2 \cap
\lambda(1)_1$ is contained in $E$. Then we have:

\begin{prop}\label{prop:thin}
For $(p,q) = (1,0)$ or $(0,1)$ the set $(D(e)_1 \cup D(e)_2 \cup
\lambda(1)_1) \smetmi E \subset L_E (p,q)$ has an $\cR_E
(p,q)$--contractible neighborhood.
\end{prop}

\begin{proof}
We begin with the case $(p,q) = (0,1)$, ie $L(p,q) = S^2 \times
S^1$. The sets $D(e)_1 \cup D(e)_2 \smetmi E$ and $\lambda(1)_1
\smetmi E$ are disjoint closed subsets of $L_E (p,q)$. Thus it
suffices to find an $\cR_E  (0,1)$--contractible neighborhood for each
one of these two sets. For $\lambda(1)_1 \smetmi E$ this is
straightforward. $\lambda(1)_1 \smetmi E$ lifts diffeomorphically to
$L_{1i} \cap H(1)_i$, $i=1,2$, and it is easy to describe a
neighborhood $W$ of $L_1\cap H(1)$ in $H(1)$ which is $\cP$--contractible in
$W$. Then the image of $W_1 \cup W_2$ is the desired neighborhood of
$\lambda(1)_1 \smetmi E$.

The neighborhood of $D(e)_1 \cup D(e)_2 \smetmi E$ will again be of
the form $W_1 \cup W_2$ where $W$ will be a neighborhood of
$D(e)\smetmi \tilde{E}$ in $H(e)$. Here $\tilde{E} \subset \partial H$
corresponds to the inverse image $\tilde{E}_i \subset \partial H_i$ of
$E\subset \partial V_1 = \partial V_2$ under the covering map. $\tilde{E}\cap D(e)$ is a
finite set of the form
$$
\{(e^{1+y_j \pi i},0): j=1,\ldots,k\}\quad\textrm{ with }\quad 
y_1 = -1 < y_2 <\cdots< y_k < 1\,.
$$
Let $W = \{(z,t) \in H(e): t>0\} \cup U\times\{0\}$ where $U\subset\C$
is the set $\partial H(e) \smetmi \smash{\bigcup_{j=1}^k L_{y_j}}$. The
restriction of the $\cP$--deformation of \fullref{lem:round} with $a=e$ to $W$ will
produce a set $W'$ with $\partial W' = \partial D(e)\smetmi\tilde{E}$,
ie, a set of $k$ disjoint intervals in $\partial D(e)$. There is a
$\cP$--deformation of $H\smetmi \smash{\bigcup_{j=1}^k L_{y_j}}$
analogous to the one of \fullref{lem:rip} which is on $\partial H$
of the form $((e^{x+iy},0),s) \longmapsto (e^{x+if(y,s,\pi)},0)$ such
that the image of $\partial H\smetmi \smash{\bigcup_{i=1}^k L_{y_i}}$
after the deformation is $\smash{\bigcup_{i=1}^k L_{z_i}}$ with $y_1 < z_1 < y_2
< z_2 <\cdots< y_k < z_k < 1$. Furthermore, this deformation is in
$H\smetmi \smash{\bigcup_{i=1}^k L_{y_i}}$, so in particular in
$H\smetmi\tilde{E}$. Applying it to $W'$ results in a set $W''$ such
that $\partial W'' = \{e^{1+i\pi z_j}: j=1,\ldots,k\}$. Choose a
contraction $h\co  \partial W'' \times I \lra \partial H\smetmi\tilde{E}$
and apply \fullref{lem:cofcontr} with $F=\partial W''$,
 to obtain a $\cP$--contraction of $W''$ in
$H\smetmi\tilde{E}$. Altogether we obtain a $\cP$--contraction of $W$
in $H\smetmi\tilde{E}$ which we apply to $W_1$ and $W_2$. Following
this deformation with the projection maps $H_i \smetmi\tilde{E}_i \lra
L_E (p,q)$ gives an $\cR_E (0,1)$ contraction of the neighborhood $W_1
\cup W_2$ of $D(e)_1 \cup D(e)_2 \smetmi E$.

For $S^3$, ie $(p,q) = (1,0)$, notice that $\lambda(1)_1 \subset
\partial D(e)_2$ and that $D(e)_1 \cap D(e)_2 \smetmi E=\varnothing$
since $\partial D(e)_1 \cap \partial D(e)_2 \in E$. Therefore it
suffices to find $\cR_E (p,q)$--categorical neighborhoods of $D(e)_1
\smetmi E$ and $D(e)_2 \smetmi E$. Both sets are meridional disks with
finitely many but at least one point removed from the boundary. Since
our treatment will work for any set of this type we only consider
$D(e)_1 \smetmi E$.

We denote $E$ when considered as a subset of $D(e)_1$ by $E_1$, and
the inverse image of $E\subset V_1$ in $H_1$ will be denoted by
$\tilde{E}_1$. Viewed as a subset of $V_2$ the set $\partial D(e)_1$
lifts to $\lambda(1)_2$ in $H_2$. The inverse image of $E$ in $H_2$
will be denoted by $\tilde{K}_2$ (In our notation $\tilde{K}_2$
differs form $\tilde{E}_2$). As in the case of $L(0,1)$ the set $E_1$
has the form
$$
\{(e^{1+y_j \pi i},0)_1 : j=1,\ldots,k\} \quad\textrm{ with }\quad
-1 = y_1 < y_2 <\cdots< y_k < 1\,.
$$
Then $K_2 = \{(e^{y_j +\pi i},0)_2 : j=1,\ldots,k\}$ is a fundamental
domain for the covering $\tilde{K}_2 \lra E$. The set $K_2$ is
contained in $\{(e^{r+\pi i},0): -1\leq r\leq 1\} \subset L_{12}$
which maps to $\partial D(e)_1 \subset \partial V_1 = \partial V_2$
under the covering map $H_2 \lra V_2$.

As before, let $W = \{(z,t) \in H(e): t>0\} \cup U\times\{0\}$ where
$U\subset\C$ is the set $\partial H(e) \smetmi \bigcup_{j=1}^k
L_{y_j}$. The inverse image of $U_1 \subset \partial V_1 = \partial
V_2$ under the diffeomorphism $H(1)_2 \lra V_2 \smetmi D(e)_2$ is the
set $(Z\times\{0\})_2$ where
$$
Z = \big\{ e^{y+a\cdot\pi i} : 0<a<2,\, y\in(-1,1) \smetmi
\{y_1,\ldots,y_k\} \big\}\,.
$$
Set $Y = \{(z,t)\in H(1): z\in Z,\, 0\leq t<e^{-1}\}$. Then $W_1 \cup
Y_2$ maps to a neighborhood of $D(e)_1 \smetmi E$ in $L_E (1,0)$. We
will identify $W_1$ and $Y_2$ with their diffeomorphic images in $V_1$
and $V_2$. Notice that $\partial W_1 = \partial Y_2$ in $L_E
(1,0)$. Therefore, any $\cP$--deformation of $W_1$ will induce a
$\cP$--deformation on $\partial Y_2$. We will extend this to a
$\cP$--deformation of $Y_2$ by mapping $(z,t)\in Y_2$ to $(z',t)\in
H_2$ if the deformation induced on $\partial Y$ maps $(z,0)$ to
$(z',0)$. For the $\cP$--contraction of $W_1$ we take the same one as
above defined for $W_1 \subset V_1 \subset L(0,1) = S^2 \times S^1$,
but there is one additional point that we have to pay attention
to. Once we have deformed $W$ to $W''$ with $\partial W'' =
\{(e^{1+i\pi z_j},0): j=1,\ldots,k\}$ we want to apply \fullref{lem:cofcontr} after choosing a contraction $h\co  \partial W''
\times I \lra \partial H\smetmi\tilde{E}$. The resulting
$\cP$--contraction of $W''_1$ induces a homotopy $k$ on the subset
$K''_2 \subset \partial H_2$, where
$$
K'' = \{(e^{z_j +\pi i},0): j=1,\ldots,k\}\,.
$$
In order that the $\cP$--deformation on $Y_2$ induced by the
$\cP$--contraction on $\partial W_1$ is a $\cP$--contraction the
homotopy $k$ has to be a contraction. This depends on the choice of
$h$. While the projection of $k$ to $\partial V$ is a contraction, $k$
itself need not be one as easy examples show. But in our situation,
for any contraction $h\co  \partial W'' \times I \lra \partial
H\smetmi\tilde{E}$ which factors through $\partial H\smetmi
(\tilde{E}\cup L_{-1}) \hookrightarrow \partial H\smetmi\tilde{E}$
the induced $k\co K'' \times I \lra \partial H\smetmi\tilde{K}$ will be
a contraction.
\end{proof}

\begin{rem}
The fact that the homotopy $k$ induced by the retraction $h$ is
sometimes not a contraction is the reason why our simple construction
can not be extended to deal with $\cR_E (p,q)$, $E\subset T^2$, for
$p>1$.
\end{rem}

\section{Foliations of category 1} 
\label{sec:cat1}

A connected surface is a $K(\pi,1)$ unless it is the $2$--sphere or
the projective plane. So our main result tells us that a $2$--dimensional
foliation on a closed manifold has category 2 unless there is a
spherical leaf. It would be nice if we could determine the category of a
$2$--dimensional foliation by simply looking at its leaves. We are not
yet in this position. But by our next result the only case that remains
open for $2$--dimensional foliations is the case of $2$--sphere bundles.

\begin{thm} \label{thm:spherebundle}
Let $\cF$ be a $p$--dimensional $C^1$--foliation of a closed
$n$--manifold $M$ with $\cat\cF\leq1$. Then $p\leq1$ or the leaves of
$\cF$ are the fibres of a homotopy-$p$--sphere bundle.
\end{thm}

\begin{proof}
Let $p$ be greater than $1$. Since the usual category, $\cat L$, of
any leaf $L$ of $\cF$ is at most $\cat\cF$ any compact leaf $L$ of
$\cF$ is a homotopy $p$--sphere. Since $p>1$, the leaf $L$ is then
$1$--connected. By the Reeb stability theorem the foliation $\cF$ near
$L$ is a product foliation. So it suffices to prove that all leaves of
$\cF$ are compact.

Let $L$ be a noncompact leaf of $\cF$ and let $\{U_0,U_1\}$ be a
tangentially categorical open cover of $M$. We may assume that both
$U_i$ are the interiors of compact triangulated submanifolds
$\bar{U}_i$ of $M$ with $\bar{U}_0 \cap \bar{U}_1 = N\times [0,1]$ and
$N\times \{i\} = \partial\bar{U}_i$, $i=0,1$. Set $M_0 = U_0\smetmi
N\times (0,\frac{1}{2})$ and $M_1 = U_1 \smetmi N\times
(\frac{1}{2},1)$, $N=N\times \{\frac{1}{2}\}$. Then $M=M_0 \cup M_1$,
$M_1 \cap M_0 =N$. By \cite[Section 5]{Thur}, we also may assume that
$N$ is in general position with respect to $\cF$ in the following
sense: $M_0$, $M_1$ are subcomplexes of a triangulation $\tau$ of $M$
which is in general position with respect to $\cF$ as defined by Thurston
\cite[Section~2]{Thur}. (See also Benameur \cite[Section~2]{Ben} for a nice rendition of Thurston's proof  given in  \cite[Section 5]{Thur}.) 

We will show below (see Lemmas \ref{lem:compcomp} and
\ref{lem:discrete}) that then the components of $L\cap M_i$ are
compact in the leaf topology of $L$ and that the set $\cC_i$ of
components of $L\cap M_i$ is discrete in the sense that each point of
$L$ has a neighborhood in $L$ which intersects at most one
$C\in\cC_i$.

Therefore, for each $C\in\cC_0$ we find a compact connected
$p$--dimensional submanifold $L_C$ of $L$ containing $C$ in its
interior and such that $\cL_0 = \{L_C : C\in\cC_0 \}$ is
discrete. Furthermore, we may assume that each $L_C$ is contained in
$U_0$. Then every boundary component of every $L_C$ is contractible in
$L$, and since $p>1$, every boundary component of every $L_C$ bounds a
$p$--manifold in $L$. Since all components of $L\smetmi \bigcup\{{\rm int}
L_C \co C\in\cC_0\}$ are closed subsets of components of $L\cap M_i$ and
are therefore compact, we find an infinite sequence $(E_i)_{i\in\N}$
of compact submanifolds of $L$ such that for all $i$ the boundary
$\partial E_i$ of $E_i$ is a boundary component of some $L_C$, $E_i$
is contained in the interior of $E_{i+1}$ and $L=\bigcup E_i$. If
$x\in {\rm int} E_1$ then no $\partial E_i$ bounds in $L\smetmi
\{x\}$. Therefore, if $h\co  U_0 \times I\lra M$ is an $\cF$--contraction,
for any $i\in\N$ there exists $y_i \in \partial E_i$, $t_i \in [0,1]$
with $h(y_i, t_i)=x$. But this is impossible since the $y_i$
eventually leave any compact subset of $L$ (because $\cL_0$ is discrete)
and since $M_0 \times [0,1]$ is compact.
\end{proof}

\begin{lem}\label{lem:compcomp}
Let $U$ be a tangentially categorical set with respect to some
foliation $\cF$ of a manifold $M$ and let $K\subset U$ be a compact
set. Then for any leaf $L$ of $\cF$ each component of $K\cap L$ is
compact in the leaf topology of $L$.
\end{lem}

\begin{proof}
Let $h\co  U\times I\lra M$ be an $\cF$--contraction and let $C$ be a
component of $K\cap L$. Then $h_1 (C)$ is a point and $C$ is a
component of the compact set $K\cap h_1^{-1} (h_1(C))$. Therefore, $C$
is a compact subset of $M$. Let $D$ be the component of $U\cap L$
containing $C$. By Proposition~1.1 of \cite{SinV} every point $x\in D$
contains arbitrarily small neighborhoods $V(x)$ in $L$ such that
$V(x)$ is contained in a neighborhood $W(x)$ of $x$ in $M$ with
$W(x)\cap D = V(x)$. Therefore, $C$ is also compact in the leaf
topology.
\end{proof}

\begin{lem}\label{lem:discrete}
Let $\cF$ be a $p$--dimensional $C^1$--foliation of the $n$--manifold $M$
and let $\tau$ be a $C^1$--triangulation of $M$ which is in general
position with respect to $\cF$. Let $M_0 \subset M$ be an
$n$--dimensional submanifold which is a subcomplex of $\tau$, and let
$L$ be a leaf of $\cF$. Then every $x\in L$ has a neighborhood $V$ in
$L$ such that $V$ intersects at most one component of $L\cap M_0$.
\end{lem}

\begin{rem}
For our proof it suffices that $\tau$ is transverse to $\cF$ as
defined in \cite[Section~2]{Thur}.
\end{rem}

\begin{proof}
Let $N=\partial M_0$. Obviously, the Lemma holds if $x\not\in N$. So
assume that $x\in N$. Let
$\sigma$ be the open simplex of $\tau$ containing $x$. By
transversality there is a neighborhood $V$ of $x$ in $L$ contained in
the open star of $\sigma$, intersecting the interior of no simplex
$\sigma'$ with $\sigma \lvertneqq \sigma'$ and $\dim \sigma' \leq
n-p$, and intersecting the interior of every simplex $\sigma'$ with
$\sigma\leq\sigma'$ and $\dim \sigma' > n-p$ either not at all or 
in a connected
manifold whose closure contains $x$. Thus $V$ intersects only the
component of $L\cap M_0$ containing $x$.
\end{proof}

\fullref{thm:spherebundle} naturally raises the following:

\begin{problem} 
Determine the tangential category of foliations whose leaves are the 
fibres of homotopy sphere bundles.
\end{problem}

Obviously this number is 1 if the bundle has a section, and by
Proposition 5.1 of \cite{SinV} it is not greater than the number of
open sets which cover the base space such that the bundle restricted
over these sets admits a section. So, in particular, for sphere
bundles over spheres it is equal to 1 or 2. The lowest dimensional
case, which is (as far as we know) unresolved, is the tangential
category of the bundle  
$$
S^2\longrightarrow \C P^3 \longrightarrow S^4.
$$

\bibliographystyle{gtart}
\bibliography{link}

\end{document}